\documentclass[a4paper,10pt]{amsart}
\usepackage{graphicx, amssymb}

\numberwithin{equation}{section}

\def\bbb {\mathbf{b}}

\def\bh {\mathbf{h}}
\def\bK {\mathbf{K}}
\def\bN {\mathbf{N}}
\def\bQ {\mathbf{Q}}
\def\bR {\mathbf{R}}
\def\bS {\mathbf{S}}
\def\bT {\mathbf{T}}
\def\bZ {\mathbf{Z}}

\def\cA {\mathcal{A}}

\def\cF {\mathcal{F}}

\def\cL {\mathcal{L}}

\def\cQ {\mathcal{Q}}
\def\cR {\mathcal{R}}

\def\cZ {\mathcal{Z}}

\def\a {{\alpha}}

\def\g {{\gamma}}
\def\de {{\delta}}
\def\eps {{\epsilon}}
\def\th {{\theta}}

\def\l {{\lambda}}
\def\si {{\sigma}}

\def\om {{\omega}}
\def\Om {{\Omega}}

\def\d {{\partial}}
\def\grad {{\nabla}}

\def\rstr {{\big |}}
\def\indc {{\bf 1}}

\def\la {\langle}
\def\ra {\rangle}

\def\wto {{\rightharpoonup}}

\newcommand{\Div}{\operatorname{div}}

\newcommand{\Supp}{\operatorname{supp}}

\newcommand{\ba}{\begin{aligned}}
\newcommand{\ea}{\end{aligned}}

\newcommand{\be}{\begin{equation}}
\newcommand{\ee}{\end{equation}}

\newcommand{\lb}{\label}

\newtheorem{Thm}{Theorem}[section]

\newtheorem{Prop}[Thm]{Proposition}
\newtheorem{Cor}[Thm]{Corollary}
\newtheorem{Lem}[Thm]{Lemma}

\begin{document}

\title[Boltzmann-Grad limit for periodic Lorentz gas]{On the Boltzmann-Grad limit \\ for the two dimensional periodic Lorentz gas}

\author[E. Caglioti]{Emanuele Caglioti}
\address[E. C.]%
{Universit\`a di Roma ``La Sapienza"\\
Dipartimento di Matematica\\
p.le Aldo Moro 5\\
00185 Roma, Italia} 

\email{caglioti@mat.uniroma1.it}

\author[F. Golse]{Fran\c cois Golse}
\address[F. G.]%
{Ecole polytechnique\\
Centre de math\'ematiques L. Schwartz\\
F91128 Palaiseau cedex\\
\& Universit\'e P.-et-M. Curie\\
Laboratoire J.-L. Lions, BP 187\\
F75252 Paris cedex 05} 

\email{francois.golse@math.polytechnique.fr}



\begin{abstract}
The two-dimensional, periodic Lorentz gas, is the dynamical system corresponding with the free motion of a point particle in a planar system of 
fixed circular obstacles centered at the vertices of a square lattice in the Euclidian plane. Assuming elastic collisions between the particle and the 
obstacles, this dynamical system is studied in the Boltzmann-Grad limit, assuming that the obstacle radius $r$ and the reciprocal mean free path 
are asymptotically equivalent small quantities, and that the particle's distribution function is slowly varying in the space variable. In this limit, the
periodic Lorentz gas cannot be described by a linear Boltzmann equation (see [F. Golse, Ann. Fac. Sci. Toulouse \textbf{17} (2008), 735--749]), 
but involves an integro-differential equation conjectured in [E. Caglioti, F. Golse, C.R. Acad. Sci. S\'er. I Math. \textbf{346} (2008) 477--482] and 
proved in [J. Marklof, A. Str\"ombergsson, preprint arXiv:0801.0612], set on a phase-space larger than the usual single-particle phase-space. The 
main purpose of the present paper is to study the dynamical properties of this integro-differential equation: identifying its equilibrium states, 
proving a H Theorem and discussing the speed of approach to equilibrium in the long time limit. In the first part of the paper, we derive the explicit 
formula for a transition probability appearing in that equation following the method sketched in [E. Caglioti, F. Golse, loc. cit.].
\end{abstract}

\maketitle


\section{The Lorentz gas}
\label{S-LorGas}

The Lorentz gas is the dynamical system corresponding with the free motion of a single point particle in a system of fixed spherical 
obstacles, assuming that collisions between the particle and any of the obstacles are elastic. This simple mechanical model was
proposed in 1905 by H.A. Lorentz \cite{Lorentz1905} to describe the motion of electrons in a metal --- see also the work of P. Drude
\cite{Drude1900}

Henceforth, we assume that the space dimension is $2$ and restrict our attention to the case of a periodic system of obstacles. 
Specifically, the obstacles are disks of radius $r$ centered at each point of $\mathbf{Z}^2$. Hence the domain left free for particle 
motion is
\begin{equation}
\label{BillTable}
Z_r=\{x\in\bR^2\,|\,\hbox{dist}(x,\bZ^2)>r\}\,,\qquad\hbox{ where $0<r<\frac12$.}
\end{equation}

Throughout this paper, we assume that the particle moves at speed $1$. Its trajectory starting from $x\in Z_r$ with velocity 
$\om\in\bS^1$ at time $t=0$ is denoted by $t\mapsto(X_r,\Om_r)(t;x,\om)\in\bR^2\times\bS^1$. One has
\begin{equation}
\label{Traj_r}
\left\{
\begin{array}{rl}
\dot{X_r}(t)&=\Om_r(t)\,,
\\
\dot\Om_r(t)&=0\,,
\end{array}
\right.
\hbox{ whenever $X_r(t)\in Z_r$,}
\end{equation}
while
\begin{equation}
\label{Traj_rColl}
\left\{
\begin{array}{rl}
X_r(t+0)&=X_r(t-0)\,,
\\
\Om_r(t+0)&=\mathcal{R}[X_r(t)]\Om_r(t-0)\,,
\end{array}
\right.
\hbox{ whenever $X_r(t)\in\d Z_r$.}
\end{equation}
In the system above, we denote $\dot{}=\frac{d}{dt}$, and $\cR[X_r(t)]$ is the specular reflection on $\d Z_r$ at the point 
$X_r(t)=X_r(t\pm 0)$. 

Next we introduce the Boltzmann-Grad limit. This limit assumes that $r\ll 1$ and that the initial position $x$ and direction $\om$ of 
the particle are jointly distributed in $Z_r\times\bS^1$ under some density of the form $f^{in}(rx,\om)$ --- i.e. slowly varying in $x$. 
Given this initial data, we define
\begin{equation}
\label{Def-f_r}
f_r(t,x,\om):=f^{in}(rX_r(-t/r;x,\om),\Om_r(-t/r;x,\om))\quad\hbox{ whenever $x\in Z_r$.}
\end{equation}
In this paper, we are concerned with the limit of $f_r$ as $r\to 0^+$ in some sense to be explained below. In the 2-dimensional 
setting considered here, this is precisely the Boltzmann-Grad limit.

In the case of a random (Poisson), instead of periodic, configuration of obstacles, Gallavotti \cite{Gallavotti} proved that the expectation 
of $f_r$ converges to the solution of the Lorentz kinetic equation 
\begin{equation}
\left\{
\label{LorentzKinEq}
\ba
{}&(\d_t\!+\!\om\!\cdot\!\grad_x)f(t,x,\om)
	=\!\int_{\bS^1}(f(t,x,\om\!-\!2(\om\cdot n)n)\!-\!f(t,x,\om))(\om\cdot n)_+dn\,,
\\
&f\Big|_{t=0}=f^{in}\,,
\ea
\right.
\end{equation}
for all $t>0$ and $(x,\om)\in\bR^2\times\bS^1$. Gallavotti's remarkable result was later generalized and improved in \cite{Spohn1978,BoldriBuniSinai1983}.

In the case of a periodic distribution of obstacles, the Boltzmann-Grad limit of the Lorentz gas cannot be described by the Lorentz 
kinetic equation (\ref{LorentzKinEq}). Nor can it be described by any linear Boltzmann equation with regular scattering kernel: 
see \cite{Golse2003,Golse2008} for a proof of this fact, based on estimates on the distribution of free path lengths to be found  in \cite{BourgainGolseWennberg} and \cite{GolseWennberg}. 

In a recent note \cite{CagliotiFG2008}, we have proposed a kinetic equation for the limit of $f_r$ as $r\to 0^+$. The striking new 
feature in our theory for the Boltzmann-Grad limit of the periodic Lorentz gas is an extended single-particle phase-space (see also
\cite{Golse2007}) where the limiting equation is posed. 

Shortly after our announcements \cite{Golse2007,CagliotiFG2008}, J. Marklof and A. Str\"ombergsson independently arrived at the same limiting 
equation for the Boltzmann-Grad limit of the periodic Lorentz gas as in \cite{CagliotiFG2008}. Their contribution \cite {MarkloStrom2008} provides
a complete rigorous derivation of that equation (thereby  confirming an hypothesis left unverified in \cite{CagliotiFG2008}), as well as an extension 
of that result to the case of any space dimension higher than 2.

The present paper provides first a complete proof of the main result in our note \cite{CagliotiFG2008}. In fact the method sketched in  
our announcement \cite{CagliotiFG2008} is different from the one used in \cite{MarkloStrom3}, and could perhaps be useful for future 
investigations on the periodic Lorentz gas in $2$ space dimensions.

Moreover, we establish some fundamental qualitative features of the equation governing the Boltzmann-Grad limit of the periodic 
Lorentz gas in $2$ space dimensions --- including an analogue of the classical Boltzmann H Theorem, a description of the equilibrium 
states, and of the long time limit for that limit equation.

We have split the presentation of our main results in the two following sections. Section 2 introduces our kinetic theory in an extended 
phase space for the Boltzmann-Grad limit of the periodic Lorentz gas in space dimension $2$. Section 3 is devoted to the fundamental 
dynamical properties of the integro-differential equation describing this Boltzmann-Grad limit --- specifically, we present an analogue
of Boltzmann's H Theorem, describe the class of equilibrium distribution functions, and investigate the long time limit of the distribution
functions in extended phase space that are solutions of that integro-differential equation.


\section{Main Results I: The Boltzmann-Grad Limit}
\label{S-BGLimit}


Let $(x,\om)\in Z_r\times\bS^1$, and define $0<t_0<t_1<\ldots$ to be the sequence of collision times on the billiard trajectory in $Z_r$ 
starting from $x$ with velocity $\om$. In other words,
\begin{equation}
\label{Def-tn}
\{t_j\,|\,j\in\bN\}=\{t\in\bR_+^*\,|\,X_r(t;x,\om)\in\d Z_r\}\,.
\end{equation}
Define further
\begin{equation}
\label{Def-xn-omn}
(x_j,\om_j):=(X_r(t_j+0;x,\om),\Om_r(t_j+0;x,\om))\,,\quad j\ge 0\,.
\end{equation}

Denote by $n_x$ the inward unit normal to $Z_r$ at the point $x\in\d Z_r$, and consider 
\begin{equation}
\ba
\Gamma_r^\pm
	&=\{(x,\om)\in\d Z_r\times\bS^1\,|\,\pm\om\cdot n_x>0\}\,,
\\
\Gamma_r^0
	&=\{(x,\om)\in\d Z_r\times\bS^1\,|\,\om\cdot n_x=0\}\,.
\ea
\end{equation}
Obviously, $(x_j,\om_j)\in\Gamma^+_r\cup\Gamma^0_r$ for each $j\ge 0$.

\subsection{The Transfer Map}
\label{SS-TMap}

As a first step in finding the Boltzmann-Grad limit of the periodic Lorentz gas, we seek a mapping from $\Gamma^+_r\cup\Gamma^0_r$ 
to itself whose iterates transform $(x_0,\om_0)$ into the sequence $(x_j,\om_j)$ defined in (\ref{Def-xn-omn}).

For $(x,\om)\in(Z_r\times\bS^1)\cup\Gamma^+_r\cup\Gamma^0_r$, let $\tau_r(x,\om)$ be the exit time defined as
\begin{equation}
\label{Def-tau}
\tau_r(x,\om)=\inf\{t>0\,|\,x+t\om\in\d Z_r\}\,.
\end{equation}
Also, for $(x,\om)\in\Gamma^+_r\cup\Gamma^0_r$, define the impact parameter $h_r(x,\om)$ as on Figure 1 by
\begin{equation}
\label{Def-h}
h_r(x,\om')=\sin(\widehat{\om',n_x})\,.
\end{equation}


\begin{figure}\lb{F-Fig1}

\begin{center}
\includegraphics[width=8.0cm]{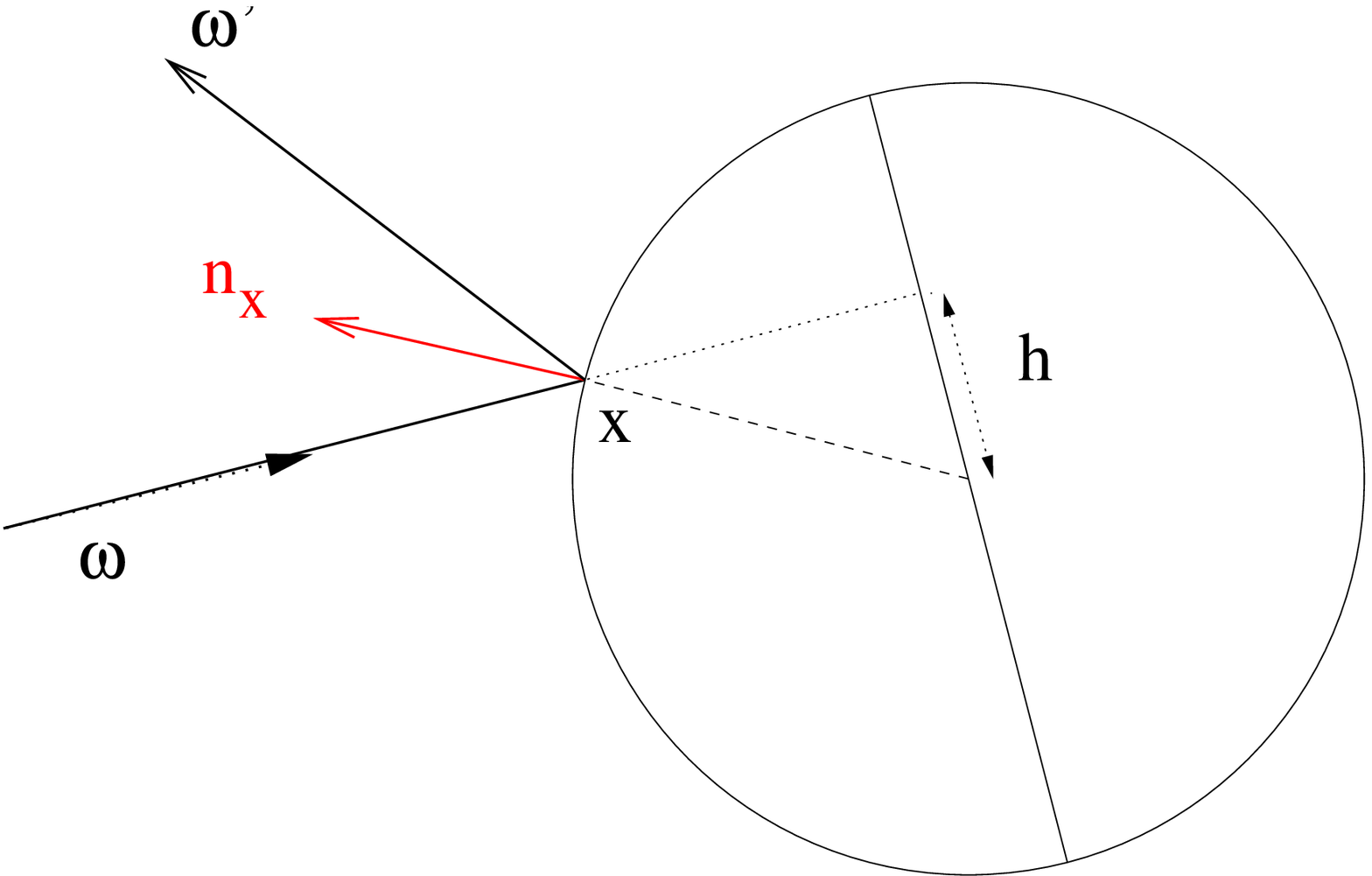}
\end{center}

\caption{The impact parameter $h$ corresponding with a collision with incoming direction $\om$ or equivalently with outgoing 
direction $\om'$}

\end{figure}


Denote by $(\Gamma^{+}_r\cup\Gamma^0_r)/\bZ^2$ the quotient of $\Gamma^{+}_r\cup\Gamma^0_r$ under the action of $\bZ^2$ 
by translations on the $x$ variable. Obviously, the map 
\begin{equation}
\label{Def-Y}
(\Gamma^{+}_r\cup\Gamma^0_r)/\mathbf{Z}^2\ni(x,\om)\mapsto(h_r(x,\om),\om)\in[-1,1]\times\bS^1
\end{equation} 
coordinatizes $(\Gamma^{+}_r\cup\Gamma^0_r)/\mathbf{Z}^2$, and we henceforth denote by $Y_r$ its inverse. For $r\in]0,\frac12[$, 
we define the transfer map (see Figure 2)
$$
T_r:\,[-1,1]\times\bS^1\to\bR_+^*\times[-1,1]
$$ 
by
\begin{equation}
\label{TransferMap}
T_r(h',\om)=(2r\tau_r(Y_r(h',\om)),h_r((X_r,\Om_r)(\tau_r(Y_r(h',\om))\pm 0;Y_r(h',\om))))\,.
\end{equation}


\begin{figure}\lb{F-Fig2}

\begin{center}
\includegraphics[width=12.0cm]{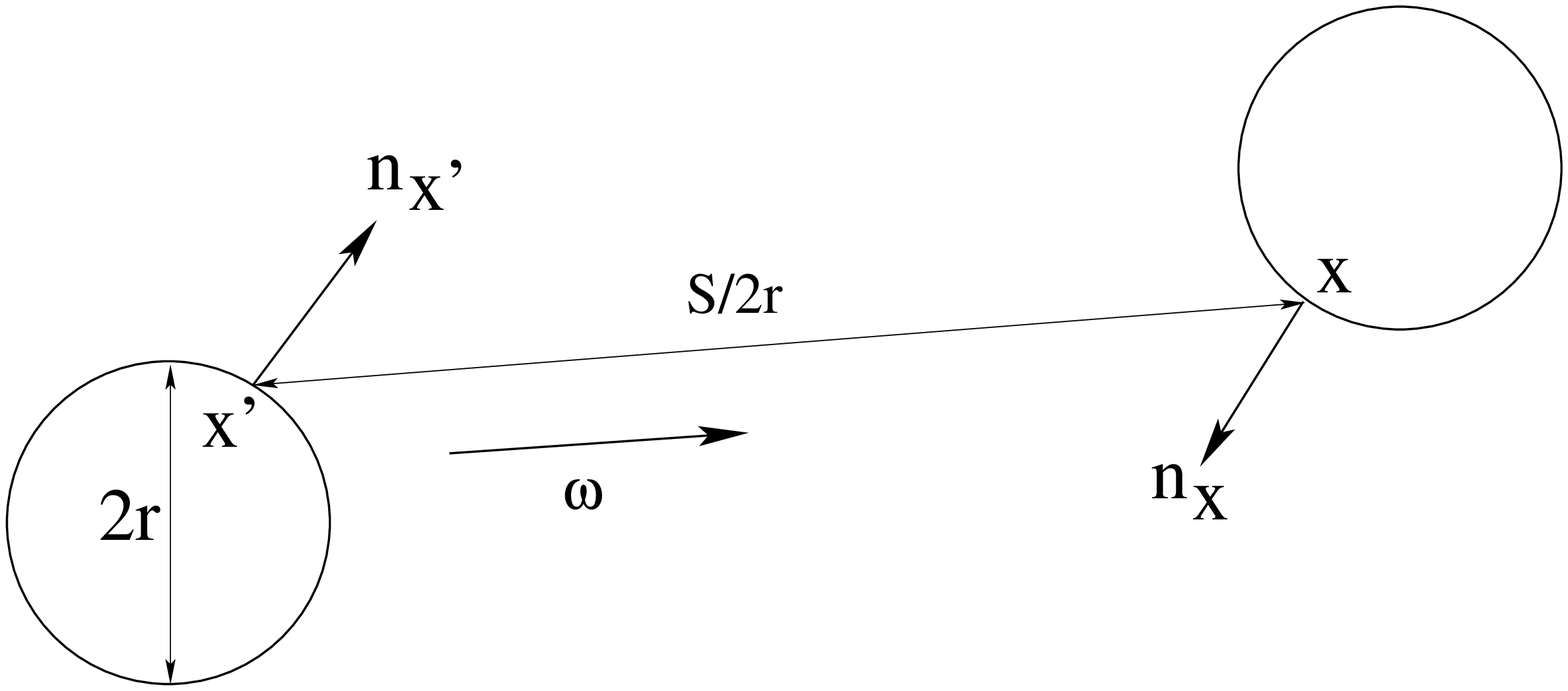}
\end{center}

\caption{The transfer map}

\end{figure}


Up to translations by a vector of $\bZ^2$, the transfer map $T_r$ is essentially the sought transformation, since one has
\begin{equation}
\label{Tr-xnomn}
T_r(h_r(x_j,\om_j),\om_j)=(2r\tau_r(x_j,\om_j),h_r(x_{j+1},\om_j))\,,\quad\hbox{ for each }j\ge 0\,,
\end{equation} 
and
\begin{equation}
\label{omn-n+1}
\omega_{j+1}=R[\pi-2\arcsin(h_r(x_{j+1},\om_j))]\om_j\,,\quad\hbox{ for each }j\ge 0\,.
\end{equation} 
The notation 
\begin{equation}
\label{Not-Rot}
R[\th]\hbox{ designates the rotation of an angle $\th$.}
\end{equation}
Notice that, by definition, 
$$
h_r(x_{j+1},\om_j)=h_r(x_{j+1},\om_{j+1})\,.
$$

The theorem below giving the limiting behavior of the map $T_r$ as $r\to 0^+$ was announced in \cite{CagliotiFG2008}.

\begin{Thm}
\label{T-LimTr}
For each $\Phi\in C_c(\bR_+^*\times[-1,1])$ and each $h'\in[-1,1]$
\begin{equation}
\label{LimTransitProba}
\frac1{|\ln\eps|}\int_\eps^{1/4}\Phi(T_r(h',\om))\frac{dr}{r}\to\int_0^\infty\int_{-1}^1\Phi(S,h)P(S,h|h')dSdh
\end{equation}
a.e. in $\om\in\bS^1$ as $\eps\to 0^+$, where the transition probability $P(S,h|h')dSdh$ is given by the formula
\begin{equation}
\ba
\label{FlaTransitProba}
P(S,h|h')=
\frac{3}{\pi^2 S\eta}\Bigl((S\eta)&\wedge(1\!-\!S)_+\!+\!\left(\eta S\!-\!|1\!-\!S|\right)_+
\\
&+\left((S-\tfrac12S\eta)\wedge(1+\tfrac12S\eta)-(\tfrac12S+\tfrac12S\zeta)\vee 1\right)_+
\\
&+\left((S-\tfrac12S\eta)\wedge 1-(\tfrac12S+\tfrac12S\zeta)\vee(1-\tfrac12 S\eta)\right)_+\Bigr)\,,
\ea
\end{equation}
with the notation
$$
\zeta=\tfrac12|h+h'|\,,\qquad\eta=\tfrac12|h-h'|\,,
$$
and 
$$
a\wedge b=\inf(a,b)\,,\qquad a\vee b=\sup(a,b)\,.
$$
Equivalently, for each $(S,h,h')\in\bR_+\times]-1,1[\times]-1,1[$ such that $|h'|\le h$, 
\begin{equation}
\label{FlaTransitProbaSimpl}
P(S,h|h')=\tfrac{3}{\pi^2}\left(1\wedge\frac{1}{h-h'}\left(\frac{2}{S}-(1+h')\right)_+\right)\,.
\end{equation}
\end{Thm}

\smallskip
The formula (\ref{FlaTransitProba}) implies the following properties of the function $P$.

\begin{Cor}[Properties of the transition probability $P(S,h|h')$.]
\lb{C-LimTr}
The function $(S,h,h')\mapsto P(S,h|h')$ is piecewise continuous on $\bR_+\times[-1,1]\times[-1,1]$.

\smallskip
\noindent
1) It satisfies the symmetries
\begin{equation}
\label{SymP}
P(S,h|h')=P(S,h'|h)=P(S,-h|-h')\quad\hbox{ for a.e. }h,h'\in[-1,1]\hbox{ and }S\ge 0\,,
\end{equation}
as well as the identities
\begin{equation}
\lb{Int-Pdsdh=1}
\left\{
\ba
{}&\int_0^\infty\int_{-1}^1P(S,h|h')dSdh=1\,,\quad\hbox{ for each }h'\in[Ð1,1]\,,
\\
&\int_0^\infty\int_{-1}^1P(S,h|h')dSdh'=1\,,\quad\hbox{ for each }h\in[Ð1,1]\,.
\ea
\right.
\end{equation}
2) The transition probability $P(S,h|h')$ satisfies the bounds
\begin{equation}
\label{BoundP}
0\le P(S,h|h')\le\frac{6}{\pi^2S}\,\mathbf{1}_{1+h'<\frac2{S}}
\end{equation}
for a.e. $h,h'\in]-1,1[$ such that $|h'|\le h$ and all $S\ge 4$. Moreover, one has
\begin{equation}
\label{BoundIntP}
\int_{-1}^1\int_{-1}^1P(S,h|h')dhdh'\le\frac{48}{\pi^2S^3}\,,\quad S\ge 4\,.
\end{equation}
\end{Cor}

As we shall see below, the family $T_r(h',\om)$ is wildly oscillating in both $h'$ and $\om$ as $r\to 0^+$, so that it is somewhat
natural to expect that $T_r$ converges only in the weakest imaginable sense.

The above result with the explicit formula (\ref{FlaTransitProba}) was announced in \cite{CagliotiFG2008}. At the same time,
V.A. Bykovskii and A.V. Ustinov\footnote{We are grateful to J. Marklof for informing us of their work in August 2009.} arrived 
independently at formula (\ref{FlaTransitProbaSimpl}) in \cite{BykovskiUstinov2009}. That formulas (\ref{FlaTransitProba}) 
and (\ref{FlaTransitProbaSimpl}) are equivalent is proved in section \ref{SS-FlaPSimpl} below.

The existence of the limit (\ref{LimTransitProba}) for the periodic Lorentz gas in any space dimension has been obtained 
by J. Marklof and A. Str\"ombergsson in \cite{MarkloStrom2007}, by a method completely different from the one used in the
work of V.A. Bykovskii and A.V. Ustinov or ours. However, at the time of this writing, their analysis does not seem to lead to 
an explicit formula for $P(s,h|h')$ such as (\ref{FlaTransitProba})-(\ref{FlaTransitProbaSimpl}) in space dimension higher
than $2$.

Notice that J. Marklof and A. Str\"ombergsson as well as V.A. Bykovskii and A.V. Ustinov obtain the limit (\ref{LimTransitProba}) 
in the weak-* $L^\infty$ topology as regards the variable $\om$, without the Ces\`aro average over $r$, whereas our result, 
being based on Birkhoff's ergodic theorem, involves the Ces\`aro average in the obstacle radius, but leads to a pointwise 
limit a.e. in $\om$.

In space dimension $2$, \cite{MarkloStrom3} extends the explicit formula (\ref{FlaTransitProba})-(\ref{FlaTransitProbaSimpl})
to the case of interactions more general than hard-sphere collisions given in terms of their scattering map. The explicit 
formula proposed by Marklof-Str\"ombergsson in  \cite{MarkloStrom3} for the transition probability follows from their formula 
(4.14) in  \cite{MarkloStrom2007}, and was obtained independently from our result in \cite{CagliotiFG2008}.

\smallskip
Another object of potential interest when considering the Boltzmann-Grad limit for the 2-dimensional periodic Lorentz gas is the
probability of transition on impact parameters corresponding with successive collisions, which is essentially the Boltzmann-Grad
limit of the billiard map in the sense of Young measures. Obviously, the probability of observing an impact parameter in some
infinitesimal interval $dh$ around $h$ for a particle whose previous collision occured with an impact parameter $h'$ is 
$$
\Pi(h|h')=\int_0^\infty P(S,h|h')dS\,.
$$

\newpage
\noindent
\fbox{\sc Explicit formula for $\Pi(h|h')$}

\smallskip
For $|h'|<h<1$, one has
\be
\lb{FlaTransitProba-hh'}
\Pi(h|h')=\tfrac{6}{\pi^2}\frac1{h-h'}\ln\frac{1+h}{1+h'}
\ee
Besides, the transition probability $\Pi(h|h')$ satisfies the symmetries inherited from $P(S,h|h')$:
\be
\label{SymPi}
\Pi(h|h')=\Pi(h'|h)=\Pi(-h|-h')\quad\hbox{ for a.e. }h,h'\in[-1,1]\,.
\ee

\subsection{3-obstacle configurations}
\label{SS-3ObstConfig}

Before analyzing the dynamics of the Lorentz gas in the Boltzmann-Grad limit, let us describe the key ideas used in our proof of 
Theorem \ref{T-LimTr}.

We begin with an observation which greatly reduces the complexity of billiard dynamics for a periodic system of obstacles centered 
at the vertices of the lattice $\bZ^2$, in the small obstacle radius limit. This observation is another form of a famous statement about
rotations of an irrational angle on the unit circle, known as ``the three-length (or three-gap) theorem", conjectured by Steinhaus
and proved by V.T. S\'os \cite{Sos1958} --- see also \cite{Suranyi1958}.

Assume $\om\in\bS^1$ has components $\om^1,\om^2$ independent over $\bQ$. Particle trajectories leaving, say, the surface of the 
obstacle centered at the origin in the direction $\om$ will next collide with one of at most three, and generically three other obstacles. 

\begin{Lem}[Blank-Krikorian \cite{BlankKriko1993}, Caglioti-Golse \cite{CagliotiFG2003}]
\label{L-3obst}
Let $0<r<\tfrac12$, and $\om\in\bS^1$ be such that $0<\om^2<\om^1$ and $\om^2/\om^1\notin\bQ$. Then, there exists $(q,p)$ 
and $(\bar q,\bar p)$ in $\bZ^2$ such that 
$$
0<q<\bar q\,,\quad q\bar p-\bar qp=\si\in\{\pm1\}
$$
satisfying the following property:
$$
\ba
{}&\{x+\tau_r(x,\om)\om\,|\,|x|=r\hbox{ and }x\cdot\om\ge 0\}
\\
&\qquad\qquad\subset
\d D((q,p),r)\cup\d D((\bar q,\bar p),r)\cup\d D((q+\bar q,p+\bar p),r)\,,
\ea
$$
where $D(x_0,r)$ designates the disk of radius $r$ centered at $x_0$.
\end{Lem}

The lemma above is one of the key argument in our analysis.

To go further, we need a convenient set of parameters in order to handle all these 3-obstacle configurations as the direction $\om$ 
runs through $\bS^1$. 

For $\om$ as in Lemma \ref{L-3obst}, the sets
$$
\{x+t\om\,|\,|x|=r\,,\,\,x\cdot\om\ge 0\,,\,\,x+\tau_r(x,\om)\om\in\partial D((q,p),r)\,,\,\,t\in\bR\}
$$
and
$$
\{x+t\om\,|\,|x|=r\,,\,\,x\cdot\om\ge 0\,,\,\,x+\tau_r(x,\om)\omega\in\partial D((\bar q,\bar p),r)\,,\,\,t\in\bR\}
$$
are closed strips, whose widths are denoted respectively by $a$ and $b$. The following quantities are somewhat easier to handle:
\begin{equation}
\label{Def-ABQQ'}
Q=\frac{2rq}{\om^1}\,,\quad \bar Q=\frac{2r\bar q}{\om^1}\,,\quad A=\frac{a}{2r}\,,\quad B=\frac{b}{2r}
\end{equation}
(see Figure 3 for the geometric interpretation of $A,B,Q$ and $\overline{Q}$), and we shall henceforth denote them by
\begin{equation}
\label{Not-ABQQ'sigma}
Q(\om,r)\,,\quad \bar Q(\om,r)\,,\quad A(\om,r)\,,\quad B(\om,r)\,,\hbox{ together with }\sigma(\om,r)
\end{equation}
whenever we need to keep track of the dependence of these quantities upon the direction $\om$ and obstacle radius $r$ --- we 
recall that
$$
\si(\om,r)=q\bar p-p\bar q\in\{\pm 1\}\,.
$$

\begin{Lem}
\label{L-RelABQQ'}
Let $0<r<\tfrac12$, and $\om\in\bS^1$ be such that $0<\om^2<\om^1$ and $\om^2/\om^1\notin\bQ$. Then, one has
\begin{equation}
\lb{BoundAB}
\left\{
\ba
{}&0<A(\om,r)\,,\,\,B(\om,r)\,,\quad A(\om,r)+B(\om,r)\le 1\,,
\\
&0<Q(\om,r)<\bar Q(\Om,r)\,,
\ea
\right.
\end{equation}
and
$$
\bar Q(\om,r)(1-A(\om,r))+Q(\om,r)(1-B(\om,r))=1\,.
$$
\end{Lem}

This last equality entails the bound
\begin{equation}
\label{BoundQ}
0<Q(\om,r)<\frac{1}{2-A(\om,r)-B(\om,r)}\le 1\,.
\end{equation}

Therefore, each possible 3-obstacle configuration corresponding with the direction $\om$ and the obstacle radius $r$ is completely
determined by the parameters $(A,B,Q,\sigma)(\om,r)\in[0,1]^3\times\{\pm1\}$.


\begin{figure}\lb{F-Fig3}

\begin{center}

\includegraphics[width=10cm]{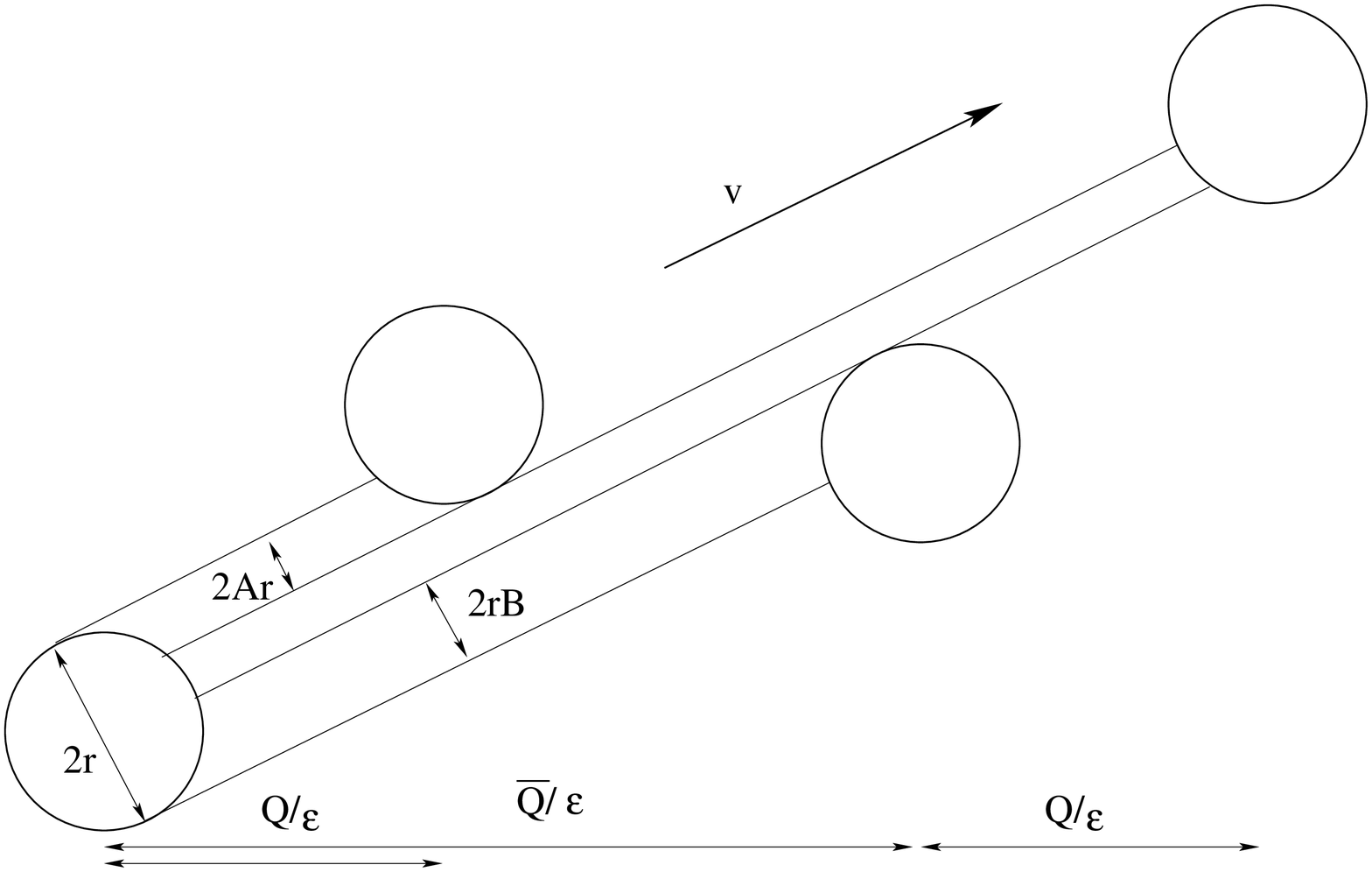}

\caption{Example of a 3-obstacle configurations, and the parameters $A$, $B$, $Q$ and $\overline{Q}$. Here $\eps=2r/\om_1$.}

\end{center}

\end{figure}


The proof of Theorem \ref{T-LimTr} is based on the two following ingredients.

The first is an asymptotic, explicit formula for the transfer map $T_r$ in terms of the parameters $A,B,Q,\sigma$ defined above. 

\begin{Prop}
\label{P-AsympTr}
Let $0<r<\tfrac12$, and $\om\in\bS^1$ be such that $0<\om^2<\om^1$ and $\om^2/\om^1\notin\bQ$. Then, for each $h'\in[-1,1]$ and 
each $r\in]0,\tfrac12[$, one has
$$
T_r(h',\om)=\mathbf{T}_{A(\om,r),B(\om,r),Q(\om,r),\si(\om,r)}(h')+(O(r^2),0)\,,
$$
in the limit as $r\to 0^+$. In the formula above, the map $\bT_{A,B,Q,\si}$ is defined for each $(A,B,Q,\si)\in[0,1]^3\times\{\pm1\}$ in 
the following manner:
\begin{equation}
\label{TransitLimit}
\begin{array}{ll}
\bT_{A,B,Q,\si}(h')=(Q,h'-2\si(1-A))&\quad\hbox{ if }\si h'\in[1-2A,1]\,,
\\
\bT_{A,B,Q,\si}(h')=\left(\overline{Q},h'+2\si(1-B)\right)&\quad\hbox{ if }\si h'\in[-1,-1+2B]\,,
\\
\bT_{A,B,Q,\si}(h')=\left(\overline{Q}+Q,h'+2\si(A-B)\right)&\quad\hbox{ otherwise.}
\end{array}
\end{equation}
\end{Prop}

For $\om=(\cos\th,\sin\th)$ with arbitrary $\th\in\bR$, the map $h'\mapsto T_r(h',\om)$ is computed using Proposition \ref{P-AsympTr}  by 
using the symmetries in the periodic configuration of obstacles as follows. Set $\tilde\th=\th-m\frac{\pi}2$ with $m=[\frac2{\pi}(\th+\frac{\pi}4)]$ 
(where $[z]$ is the integer part of $z$) and let $\tilde\om=(\cos\tilde\th,\sin\tilde\th)$. Then
\begin{equation}
\label{SymT_r}
T_r(h',\om)=(s,h)\,,\quad\hbox{ where }(s,\hbox{sign}(\tan\tilde\th)h)=T_r(\hbox{sign}(\tan\tilde\th)h',\tilde\om)\,.
\end{equation}

\smallskip
The second ingredient in our proof of Theorem \ref{T-LimTr} is an explicit formula for the limit of the distribution of 
$\om\mapsto (A,B,Q,\sigma)(\om,r)$ as $r\to 0^+$ in the sense of Ces\`aro on the first octant $\bS^1_+$ of $\bS^1$. 

\begin{Prop}\label{P-3ObstProba}
Let $F$ be any bounded and piecewise continuous function defined on the compact $\bK=[0,1]^3\times\{\pm1\}$. Then
\begin{equation}
\label{Lim3ObstParam}
\ba
\frac1{|\ln\eps|}\int_\eps^{1/4}&F(A(\om,r),B(\om,r),Q(\om,r),\si(\om,r))\frac{dr}{r}
\\
&\qquad\qquad\qquad\to\int_{\mathbf{K}}F(A,B,Q,\si)d\mu(A,B,Q,\si)
\ea
\end{equation}
a.e. in $\om\in\bS^1_+$ as $\eps\to 0^+$, where $\mu$ is the probability measure on $\bK$ given by
\begin{equation}
\label{3ObstProba}
d\mu(A,B,Q,\si)\!=
	\!{\frac{6}{\pi^2}}\indc_{0<A<1}\mathbf{1}_{0<B<1-A}\indc_{0<Q<\frac1{2-A-B}}\frac{dAdBdQ}{1-A}(\de_{\si=1}+\de_{\si=-1})\,.
\end{equation}
\end{Prop}

This result is perhaps more transparent when stated in terms of the new parameters $A,B'=\frac{B}{1-A},Q,\sigma$ instead of the 
original $A,B,Q,\sigma$: an elementary change of variables in the integral on the right hand side of (\ref{Lim3ObstParam}) shows 
that
$$
\ba
\int_{\bK}&F(A,B,Q,\si)d\mu(A,B,Q,\si)
\\
&\qquad\qquad\qquad
=
\int_{\bK}F(A,(1-A)B',Q,\sigma)d\mu'(A,B',Q,\si)\,,
\ea
$$
where
$$
\ba
{}&d\mu'(A,B',Q,\sigma)
\\
&\qquad
\!=\!{\frac{6}{\pi^2}}\indc_{0<A<1}\indc_{0<B'<1}\indc_{0<Q<\frac1{1+(1-A)(1-B')}}dAdB'dQ(\de_{\si=1}+\de_{\si=-1})\,.
\ea
$$
In other words, the new parameters $A,B',Q$ and $\si$ are uniformly distributed over the maximal domain compatible with the bounds 
(\ref{BoundAB}) and (\ref{BoundQ}).

The first part of Theorem \ref{T-LimTr} follows from combining the two propositions above; in particular,  for each $h'\in[-1,1]$,  the transition 
probability $P(S,h|h')dSdh$ is obtained as the image of the probability measure $\mu$ in (\ref{3ObstProba}) under the transformation $(A,B,Q,\si)
\mapsto\bT_{A,B,Q,\si}(h')$.

\subsection{The Limiting Dynamics}
\label{SS-LimDyn}

With the parametrization of all 3-obstacle configurations given above, we return to the problem of describing the Boltzmann-Grad limit 
of the Lorentz gas dynamics.

Let $(x,\om)\in Z_r\times\bS^1$, and let the sequence of collision times $(t_j)_{j\ge 0}$, collision points $(x_j)_{j\ge 0}$ and post-collision 
velocities $(\om_j)_{j\ge 0}$ be defined as in (\ref{Def-tn}) and (\ref{Def-xn-omn}). The particle trajectory starting from $x$ in the direction 
$\om$ at time $t=0$ is obviously completely defined by these sequences. 

As suggested above, the sequences $(t_j)_{j\ge 0}$, $(x_j)_{j\ge 0}$ and $(\om_j)_{j\ge 0}$ can be reconstructed with the transfer map, 
as follows.

Set
\begin{equation}
\label{InitInduc}
\ba
t_0&=\tau_r(x,\omega)\,,
\\
x_0&=x+\tau_r(x,\omega)\omega\,,
\\
h_0&=h_r(x_0,\omega)\,,
\\
\omega_0&=R[\pi-2\arcsin(h_0)]\omega\,.
\ea
\end{equation}
We then define the sequences $(t_j)_{j\ge 0}$, $(x_j)_{j\ge 0}$ inductively, in the following manner:
\begin{equation}
\label{Induc-shtxom}
\ba
(2s_{j+1},h_{j+1})&=T_r(h_j,\om_j)\,,
\\
t_{j+1}&=t_j+\frac1{r}s_{j+1}\,,
\\
x_{j+1}&=x_j+\frac1{r}s_{j+1}\om_j\,,
\\
\om_{j+1}&=R[\pi-2\arcsin(h_{j+1})]\om_j\,,
\ea
\end{equation}
for each $j\ge 0$.

If the sequence of 3-obstacle configuration parameters 
$$
\bbb_j^r=((A,B,Q,\sigma)(\om_j,r))_{j\ge 0}
$$
converges (in some sense to be explained below) as $r\to 0^+$ to a sequence of independent random variables $(\bbb_j)_{j\ge 0}$ 
with values in $\bK$, then the dynamics of the periodic Lorentz gas in the Boltzmann-Grad limit can be  described in terms of the 
discrete time Markov process defined as
$$
(S_{j+1},H_{j+1})=\bT_{\bbb_j}(H_j)\,,\quad j\ge 0\,.
$$

Denote $\cZ_{j+1}=\bR^2\times\bS^1\times\bR_+\times[-1,1]\times\bK^{n+1}$ for each $j\ge 0$. The asymptotic independence above 
can be formulated as follows: there exists a probability measure $P_0$ on $\bR_+\times[-1,1]$ such that, for each $j\ge 0$ and 
$\Psi\in C(\cZ_{n+1})$,
$$
\ba
\lim_{r\to 0^+}\int_{rZ_r\times\bS^1}&\Psi(x,\om,r\tau_r({\tfrac{x}r},\om),h_r({\tfrac{x_0}r},\om_0),\bbb_0^r,\ldots,\bbb_n^r)dxd\om
\\
&=
\int_{\mathcal{Z}_{n+1}}\Psi(x,\om,\tau,h,\bbb_0,\ldots,\bbb_n)dxd\om dP_0(\tau,h)d\mu(\bbb_0)\ldots d\mu(\bbb_n)\,,
\ea
\leqno(H)
$$
where $\mu$ is the measure defined in (\ref{3ObstProba}).

This scenario for the limiting dynamics is confirmed by the following 

\begin{Thm}\label{T-LimEq}
Let $f^{in}$ be any continuous, compactly supported probability density on $\bR^2\times\bS^1$. Denoting by $R[\th]$ the rotation of an 
angle $\th$, let $F\equiv F(t,x,\om,s,h)$ be the solution of
\begin{equation}
\label{EqLim}
\left\{
\ba
{}&(\d_t+\om\cdot\grad_x-\d_s)F(t,x,\omega,s,h)=\int_{-1}^12P(2s,h|h')F(t,x,R[\th(h')]\om,0,h')dh'\,,
\\
&F(0,x,\om,s,h)=f^{in}(x,\om)\int_{2s}^\infty\int_{-1}^1P(\tau,h|h')dh'd\tau\,,
\ea
\right.
\end{equation}
where $(x,\om,s,h)$ runs through $\bR^2\times\bS^1\times\bR^*_+\times]-1,1[$, and $\th(h)=\pi-2\arcsin(h)$. 

Then the family $(f_r)_{0<r<\frac12}$ defined in (\ref{Def-f_r}) satisfies
\begin{equation}
\label{LimBG}
f_r\to\int_0^\infty\int_{-1}^1F(\cdot,\cdot,\cdot,s,h)dsdh\hbox{ in $L^\infty(\bR_+\times\bR^2\times\bS^1)$ weak-$*$}
\end{equation}
as $r\to 0^+$.
\end{Thm}

Let us conclude this presentation of our main results with a few remarks. 

Equation (\ref{EqLim})-(\ref{LimBG}) was proposed first in \cite{CagliotiFG2008}, under some additional decorrelation assumption left 
unverified --- specifically, assuming (H). Then, Marklof-Str\"ombergsson provided a complete, rigorous derivation of that same equation 
in \cite{MarkloStrom2008}, without any additional assumption, thereby establishing the theorem above.

The main novelty in this description of the Boltzmann-Grad limit of the periodic Lorentz gas is the fact that it involves a Markov process 
in the extended phase space $\bR^2\times\bS^1\times\bR_+\times[-1,1]$. In addition to the space and velocity variables $x$ and $\om$ 
that are usual in the classical kinetic theory of gases, this extended phase space involves two extra variables: i.e. $s$, the (scaled) time 
to the next collision and $h$, the impact parameter at that next collision, as additional coordinates describing the state of the moving 
point particle. To the best of our knowledge, the idea of using this extended phase space (and particularly the additional variables $s$
and $h$) appeared for the first time in our announcements \cite{Golse2007,CagliotiFG2008}.


\section{Main Results II: Dynamical Properties of the Limiting Equation}
\label{S-PropLimit}

The present section establishes some fundamental mathematical properties of equation (\ref{EqLim}). For simplicity, we henceforth restrict our attention to the case where the space variable $x$ varies in the flat $2$-torus $\bT^2=\bR^2/\bZ^2$.

\subsection{Equilibrium states}
\label{SS-EquilState}

As is well-known, in the kinetic theory of gases, the equilibrium states are the uniform Maxwellian distributions. They are characterized 
as the only distribution functions that are independent of the space variable and for which the collision integral vanishes identically. 

In equation (\ref{EqLim}), the analogue of the Boltzmann collision integral is the quantity
$$
\int_{-1}^12P(2s,h|h')F(t,x,R[\th(h')]\om,0,h')dh'+\partial_sF(t,x,\om,s,h)\,.
$$
On the other hand, the variables $(s,h)$ play in equation (\ref{EqLim}) the same role as the velocity variable in classical kinetic theory.

Therefore, the equilibrium distributions analogous to Maxwellians in the kinetic theory of gases are the nonnegative measurable functions $F\equiv F(s,h)$ such that
$$
-\d_sF(s,h)=\int_{-1}^12P(2s,h|h')F(0,h')dh'\,,\quad s>0\,,\,\,-1<h<1\,.
$$

\begin{Thm}\lb{T-PropEquil}
Define 
$$
E(s,h):=\int_{2s}^{+\infty}\int_{-1}^1P(\tau,h|h')dh'd\tau\,.
$$
1) Then
$$
E(0,h)=1\,,\quad -1<h<1
$$
and 
$$
\int_0^{+\infty}\int_{-1}^1E(s,h)dhds=\int_0^{+\infty}\int_{-1}^1\int_{-1}^1\tfrac12SP(S,h|h')dhdh'dS=1\,.
$$
2) Let $F\equiv F(s,h)$ be a bounded, nonnegative measurable function on $\bR_+\times[-1,1]$ such that $s\mapsto F(s,h)$ is
continuous on $\bR_+$ for a.e. $h\in[-1,1]$ and
$$
\left\{
\ba
{}&-\d_sF(s,h)=\int_{-1}^12P(2s,h|h')F(0,h')dh'\,,\quad s>0\,,\,\,-1<h<1\,,
\\
&\lim_{s\to+\infty}F(s,h)=0\,.
\ea
\right.
$$
Then there exists $C\ge 0$ such that 
$$
F(s,h)=CE(s,h)\,,\quad s>0\,,\,\,-1<h<1\,.
$$
3) Define\footnote{The existence of this limit, and an explicit formula for $p$ are obtained in \cite{BocaZaharescu}.}
$$
p(t)=\lim_{r\to 0^+}\frac{|\{(x,\om)\in(Z_r\cap[0,1]^2)\times\bS^1\,|\,2r\tau_r(x,\om)>t\}|}{|(Z_r\cap[0,1]^2)\times\bS^1|}\,.
$$
Then
$$
\int_{-1}^1E(s,h)dh=-2p'(2s)\,,\quad s>0\,,
$$
and
$$
\int_{-1}^1E(s,h)dh\sim\frac{1}{\pi^2s^2}\hbox{ as }s\to+\infty\,.
$$
\end{Thm}

Notice that the class of physically admissible initial data for our limiting equation (\ref{EqLim}) consists of densities of the form
$$
F^{in}(x,\om,s,h)=f^{in}(x,\om)E(s,h)
$$
--- see Theorem \ref{T-LimEq}. In other words, physically admissible initial data are ``local equilibrium densities", i.e. equilibrium 
densities in $(s,h)$ modulated in the variables $(x,\om)$.

\smallskip
Before going further, we need some basic facts about the evolution semigroup defined by the Cauchy problem (\ref{EqLim}).
The existence and uniqueness of a solution of the Cauchy problem (\ref{EqLim}) presents little difficulty. It is written in the form
$$
F(t,\cdot,\cdot,\cdot)=K_tF^{in}\,,\quad t\ge 0\,,
$$
where $(K_t)_{t\ge 0}$ is a strongly continuous linear contraction semigroup on the Banach space 
$L^1(\bT^2\times\bS^1\times\bR_+\times[-1,1])$. It satisfies in particular the following properties:

\smallskip
\noindent
1) if $F^{in}\ge 0$ a.e. on $\bT^2\times\bS^1\times\bR_+\times[-1,1]$, then, for each $t\ge 0$, one has
$K_tF^{in}\ge 0$ a.e. on $\bT^2\times\bS^1\times\bR_+\times[-1,1]$;

\smallskip
\noindent
2) for each $t\ge 0$, one has $K_tE=E$;

\smallskip
\noindent
3) if $F^{in}\le CE$ (resp. $F^{in}\ge CE$) a.e. on $\bT^2\times\bS^1\times\bR_+\times[-1,1]$ for some constant $C$, then, for each 
$t\ge 0$, one has $K_tF^{in}\le CE$ (resp. $K_tF^{in}\ge CE$) a.e. on $\bT^2\times\bS^1\times\bR_+\times[-1,1]$;

\smallskip
\noindent
4) for each $F^{in}\in L^1(\bT^2\times\bS^1\times\bR_+\times[-1,1])$ and each $t\ge 0$, one has
$$
\int\!\!\!\iiint_{\bT^2\times\bS^1\times\bR_+\times[-1,1]}K_tF^{in}dxd\om dsdh
=
\int\!\!\!\iiint_{\bT^2\times\bS^1\times\bR_+\times[-1,1]}F^{in}dxd\om dsdh\,;
$$

\smallskip
\noindent
5) if $F^{in}\in C(\bT^2\times\bS^1\times\bR_+\times[-1,1])$ is continuously differentiable with respect to $x$ and $s$, i.e. 
$\grad_xF^{in}$ and $\d_sF^{in}\in C(\bT^2\times\bS^1\times\bR_+\times[-1,1])$, then, for each $t\ge 0$, one has 
$K_tF^{in}$, $\d_tK_tF^{in}$, $\grad_xK_tF^{in}$ and $\d_sK_tF^{in}\in C(\bT^2\times\bS^1\times\bR_+\times[-1,1])$, and 
the function $(t,x,\om,s,h)\mapsto K_tF^{in}(x,\om,s,h)$ is a classical solution of the equation (\ref{EqLim}) on 
$\bR_+^*\times\bT^2\times\bS^1\times\bR_+^*\times[-1,1]$.

\smallskip
All these properties follow from straightforward semigroup arguments once (\ref{EqLim}) is established. Otherwise, the semigroup 
$(K_t)_{t\ge 0}$ is constructed together with the underlying Markov process in section 6 of \cite{MarkloStrom2008} --- see in 
particular Propositions 6.2 and 6.3, formula (6.16) and Theorem 6.4 there.

\subsection{Instability of modulated equilibrium states}
\label{SS-InstabEquilState}

A well-known feature of the kinetic theory for monatomic gases is that generically, local equilibrium distribution functions --- i.e. 
distribution functions that are Maxwellian in the velocity variable and whose pressure, bulk velocity and temperature may 
depend on the time and space variables --- are solutions of the Boltzmann equation if and only if they are uniform equilibrium
distribution functions --- i.e. independent of the time and space variables. In other words, the class of local Maxwellian states
is generically unstable under the dynamics of the Boltzmann equation. An obvious consequence of this observation is that
rarefied gas flows are generically too complex to be described by only the macroscopic fields used in classical gas dynamics
--- i.e. by local Maxwellian distribution functions parametrized by a pressure, temperature and velocity field. 

Equation (\ref{EqLim}) governing the Boltzmann-Grad limit  of the periodic Lorentz gas satisfies the following, analogous
property.

\begin{Thm}\label{T-InstabLocEquil}
Let $F$ be a solution of (\ref{EqLim}) of the form
$$
F(t,x,\om,s,h)=f(t,x,\om)E(s,h)\,,\quad\hbox{ with }f\in C^1(I\times\bT^2\times\bS^1)\,,
$$
where $I$ is any interval of $\bR_+$ with nonempty interior. Then $f$ is a constant.
\end{Thm}

Thus, the complexity of the equation (\ref{EqLim}) posed in the extended phase space $\bT^2\times\bS^1\times\bR_+\times[-1,1]$
cannot be reduced by postulating that the solution is a local equilibrium, whose additional variables $s$ and $h$ can be averaged
out.

As in the case of the classical kinetic theory of gases, this observation is important in the discussion of the long time limit of solutions
of (\ref{EqLim}). 

\subsection{H Theorem and a priori estimates}
\label{SS-Apriori}

In this section, we propose a formal derivation of a class of a priori estimates that includes an analogue of Boltzmann's H Theorem
in the kinetic theory of gases.

Let $\bh$ be a convex $C^1$ function defined on $\bR_+$; consider the relative entropy 
$$
\ba
{}&H_\bh(fE|E):=
\\
&\int\!\!\!\iint_{\bT^2\times\bS^1\times\bR_+\times[-1,1]}\left(\bh(f)-\bh(1)-\bh'(1)(f-1)\right)(t,x,\om,s,h)E(s,h)dxd\om dsdh\,.
\ea
$$

The most classical instance of such a relative entropy corresponds with the choice $\bh(z)=z\ln z$: in that case $\bh(1)=0$ while 
$\bh'(1)=1$, so that
$$
H_{z\ln z}(fE|E)=\int\!\!\!\iint_{\bT^2\times\bS^1\times\bR_+\times[-1,1]}\left(f\ln f\!-\!f\!+\!1\right)(t,x,\om,s,h)E(s,h)dxd\om dsdh\,.
$$

\begin{Thm}\lb{T-HThm}
Let $F\equiv F(t,x,\om,s,h)$ in $C^1(\bR_+\times\bT^2\times\bS^1\times\bR_+\times[-1,1])$ be such that
$$
0\le F(t,x,\om,s,h)\le CE(s,h)\,,\quad  (t,x,\om,s,h)\in\bR_+\times\bT^2\times\bS^1\times\bR_+\times[-1,1]\,,
$$
and
$$
(\d_t+\om\cdot\grad_x-\d_s)F(t,x,\omega,s,h)=\int_{-1}^12P(2s,h|h')F(t,x,R[\th(h')]\om,0,h')dh'\,,
$$
with the notations of Theorem \ref{T-LimEq}. Then $H_\bh(F|E)\in C^1(\bR_+)$ and
$$
\frac{d}{dt}H_\bh(F|E)+\int_{\bT^2}D_\bh(F/E)(t,x)dx=0\,,
$$
where the entropy dissipation rate $D_\bh$ is given by the formula
$$
\ba
{}&D_\bh(f)(t,x)=
\\
&\int\!\!\!\iiint_{\bS^1\times\bR_+\times[-1,1]\times[-1,1]}\!2P(2s,h|h')\Big(\bh(f(t,x,R[\th(h')]\om,0,h'))\!-\!\bh(f(t,x,\om,s,h))
\\
&\qquad\qquad-\bh'(f(t,x,\om,s,h)\left(f(t,x,R[\theta(h')]\om,0,h')-f(t,x,\om,s,h)\right)\Big)d\om dsdhdh'\,.
\ea
$$
\end{Thm}

Integrating the equality above over $[0,t]$, one has
$$
H_\bh(F|E)(t)+\int_0^t\int_{\bT^2}D_\bh(F/E)(\tau,x)dxd\tau=H_\bh(F|E)(0)
$$
for each $t\ge 0$. Since $\bh$ is convex, one has 
$$
H_\bh(F|E)\ge 0\hbox{ and }D_\bh(F/E)\ge 0\,,
$$
and the equality above entails the a priori estimates
$$
\left\{
\ba
{}&0\le H_\bh(F|E)(t)\le H_\bh(F|E)(0)\,,
\\
\\
&\int_0^{+\infty}\int_{\bT^2}D_\bh(F/E)(t,x)dxdt\le H_\bh(F|E)(0)\,.
\ea
\right.
$$

That $H_\bh(F|E)$ is a nonincreasing function of time is a general property of Markov processes; see for instance Yosida
\cite{YosidaFuncAn} on p. 392.

\subsection{Long time limit}
\label{SS-LongTime}

As an application of the analogue of Boltzmann's H Theorem presented in the previous section, we investigate the asymptotic
behavior of solutions of (\ref{EqLim}) in the limit as $t\to+\infty$.

\begin{Thm}\lb{T-LongTime}
Let $f^{in}\equiv f^{in}(x,\om)\in L^\infty(\bT^2\times\bS^1)$ satisfy $f^{in}(x,\om)\ge 0$ a.e. in $(x,\om)\in\bT^2\times\bS^1$.
Let $F$ be the solution of the Cauchy problem (\ref{EqLim}). Then
$$
F(t,\cdot,\cdot,\cdot,\cdot)\wto CE
$$
in $L^\infty(\bT^2\times\bS^1\times\bR_+\times[-1,1])$ weak-*, with
$$
C=\tfrac1{2\pi}\iint_{\bT^2\times\bS^1}f^{in}(x,\om)dxd\om\,.
$$
\end{Thm}

\subsection{Speed of approach to equilibrium}
\label{SS-SpeedEquil}

The convergence to equilibrium in the long time limit established in the previous section may seem rather unsatisfying. Indeed, in most
cases, solutions of linear kinetic models converge to equilibrium in a strong $L^2$ topology, and often satisfy some exponential decay
estimate. 

While the convergence result in Theorem \ref{T-LongTime} might conceivably be improved, the following result rules out the possibility 
of a return to equilibrium at exponential speed in the strong $L^2$ sense.

\begin{Thm}\lb{T-NoSpecGap}
There does not exist any function $\Phi\equiv\Phi(t)$ satisfying 
$$
\Phi(t)=o(t^{-3/2})\quad\hbox{ as }t\to+\infty
$$ 
such that, for each $f^{in}\in L^2(\bT^2\times\bS^1)$, the solution $F$ of the Cauchy problem (\ref{EqLim}) satisfies the bound
\be
\lb{ExpEst}
\ba
\left\|F(t,\cdot,\cdot,\cdot,\cdot)-\la f^{in}\ra E\right\|&_{L^2(\bT^2\times\bS^1\times\bR_+\times[-1,1])}
\\
&\le\Phi(t)\|F(0,\cdot,\cdot,\cdot,\cdot)\|_{L^2(\bT^2\times\bS^1\times\bR_+\times[-1,1])}
\ea
\ee
for each $t\ge 0$, with the notation
$$
\la\phi\ra=\tfrac1{2\pi}\iint_{\bT^2\times\bS^1}\phi(x,\om)dxd\om
$$
for each $\phi\in L^1(\bT^2\times\bS^1)$.
\end{Thm}

By the same argument as in the proof of Theorem \ref{T-NoSpecGap}, one can establish a similar result for initial data in 
$L^p(\bT^2\times\bS^1)$, with the $L^2$ norm replaced with the $L^p$ norm in (\ref{ExpEst}), for all $p\in]1,\infty[$; in that case 
$\Phi(t)=o(t^{-(2p-1)/p})$ is excluded.

The case $p=2$ discussed in the theorem excludes the possibility of a spectral gap for the generator of the semigroup $K_t$ associated 
with equation (\ref{EqLim}) --- that is to say, for the unbounded operator $\cA$ on $L^2(\bT^2\times\bS^1\times\bR_+\times[-1,1])$ defined 
by
$$
\cA f(x,\om,s,h)=(\om\cdot\grad_x-\d_s)f(x,\om,s,h)-\int_{-1}^12P(2s,h|h')f(x,R[\th(h')]\om,0,h')dh'\,,
$$
with domain
$$
D(\cA)=\{f\in L^2(\bT^2\times\bS^1\times\bR_+\times[-1,1])\,|\,(\om\cdot\grad_x-\d_s)f\in L^2(\bT^2\times\bS^1\times\bR_+\times[-1,1])\}\,.
$$


\section{An ergodic theorem with continued fractions}
\label{S-ErgoContFrac}

\subsection{Continued fractions}
\label{SS-ContFrac}

Let $\a\in(0,1)\setminus\bQ$; its continued fraction expansion is denoted
\begin{equation}
\label{ContFrac-alpha}
\a=[0;a_1,a_2,a_3,\ldots]=\frac1{\displaystyle a_1+\frac1{\displaystyle a_2+\frac1{a_3+\ldots}}}\,.
\end{equation}
Consider the Gauss map
\begin{equation}
\label{GaussMap}
T:\,(0,1)\setminus\bQ\ni x\mapsto\frac1x-\left[\frac1x\right]\in(0,1)\in(0,1)\setminus\bQ\,.
\end{equation}
The positive integers $a_1,a_2,a_3,\ldots$ are expressed in terms of $\a$ as
\begin{equation}
\label{Fla-an}
a_1=\left[\frac1{\a}\right]\,,\quad\hbox{ and }\quad a_n=\left[\frac1{T^{n-1}\a}\right]\,,\quad n\ge 1\,.
\end{equation}
The action of $T$ on $\a$ is most easily read on its continued fraction expansion:
\begin{equation}
\label{T-ContFrac}
T[0;a_1,a_2,a_3,\ldots]=[0;a_2,a_3,a_4,\ldots]\,.
\end{equation}
We further define two sequences of integers $(p_n)_{n\ge 0}$ and $(q_n)_{n\ge 0}$ by the following induction procedure:
\begin{equation}
\label{Induct-pnqn}
\ba
p_{n+1}&=a_np_n+p_{n-1}\,,\quad &&p_0=1\,,\,\,\,&&p_1=0\,,
\\
q_{n+1}&=a_nq_n+q_{n-1}\,,\quad &&q_0=0\,,\,\,\,&&q_1=1\,.
\ea
\end{equation}
The sequence of rationals $(\frac{p_n}{q_n})_{n\ge 1}$ converges to $\a$ as $n\to\infty$. Rather than the usual distance
$|\frac{p_n}{q_n}-\a|$, it is more convenient to consider
\begin{equation}
\label{Def-dn}
d_n:=|q_n\a-p_n|=(-1)^{n-1}(q_n\alpha-p_n)\,,\quad n\ge 0\,.
\end{equation}
Obviously
\begin{equation}
\label{Induct-dn}
d_{n+1}=-a_nd_n+d_{n-1}\,,\quad d_0=1\,,\,\,d_1=\a\,.
\end{equation}
We shall use the notation
$$
a_n(\a)\,,\,\,\,p_n(\a)\,,\,\,\,q_n(\a)\,,\,\,\,d_n(\a)\,,
$$
whenever we need to keep track of the dependence of those quantities upon $\a$.  For each $\a\in(0,1)\setminus\bQ$, one has the 
relation $a_n(T\a)=a_{n+1}(\a)$, which follows from (\ref{T-ContFrac}) and implies in turn that $\a d_n(T\a)=d_{n+1}(\a)$, for each 
integer $n\ge 0$, by (\ref{Induct-dn}). Therefore,
\begin{equation}
\label{Fla-dn}
d_n(\a)=\prod_{k=0}^{n-1}T^k\a\,,\quad n\ge 0\,.
\end{equation}
While $a_n(\a)$ and $d_n(\a)$ are easily expressed in terms of the sequence $(T^k\a)_{k\ge 0}$, the analogous expression for 
$q_n(\a)$ is somewhat more involved. With (\ref{Induct-pnqn}) and (\ref{Induct-dn}), one proves by induction that
\begin{equation}
\label{qndn}
q_n(\a)d_{n+1}(\a)+q_{n+1}(\a)d_n(\a)=1\,,\quad n\ge 0\,.
\end{equation}
Hence
$$
q_{n+1}(\a)d_n(\a)=1-\frac{d_{n+1}(\a)d_n(\a)}{d_n(\a)d_{n-1}(\a)}q_n(\a)d_{n-1}(\a)\,,
$$
so that, by a straightforward induction
\begin{equation}
\label{Fla-qn}
q_{n+1}(\a)d_n(\a)=\sum_{j=0}^n(-1)^{n-j}\frac{d_{n+1}(\a)d_n(\a)}{d_{j+1}(\a)d_j(\a)}\,,\quad n\ge 0\,.
\end{equation}
Replacing $(d_k(\a))_{0\le k\le n+1}$ by its expression (\ref{Fla-dn}) leads to an expression of $q_{n+1}(\a)$ in terms of $T^k\a$ for 
$0\le k\le n$.

\subsection{The ergodic theorem}
\label{SS-Ergodic}

We recall that the Borel probability measure $dG(x)=\tfrac1{\ln 2}\frac{dx}{1+x}$ on $(0,1)$ is invariant under the Gauss map $T$, and 
that $T$ is ergodic for the measure $dG(x)$ (see for instance \cite{Khinchin1964}), and even strongly mixing (see \cite{Philipp1967}.)

For each $\a\in(0,1)\setminus\bQ$ and $\eps\in(0,1]$, define
\begin{equation}
\label{Def-N}
N(\a,\eps)=\inf\{n\ge 0\,|\,d_n(\a)\le\eps\}\,.
\end{equation}

\begin{Lem}
\label{L-EquivN}
For a.e. $\a\in(0,1)$, one has
$$
N(\a,\eps)\sim\tfrac{12\ln 2}{\pi^2}\ln\frac1\eps\,,\quad\eps\to 0^+\,.
$$
\end{Lem}

See \cite{CagliotiFG2003} (where it is stated as Lemma 3.1) for a proof.

We further define
\begin{equation}
\label{Def-deltan}
\de_n(\a,\eps)=\frac{d_n(\a)}{\eps}\,,\quad n\ge 0\,,
\end{equation}
for each $\a\in(0,1)\setminus\bQ$ and $\eps\in(0,1]$.

\begin{Thm}
\label{T-Ergo}
For $m\ge 0$, let $f$ be a bounded measurable function on $(\bR_+)^{m+1}$; then there exists $L_m(f)\in\bR$ independent of $\a$ 
such that
$$
\frac1{\ln(1/\eta)}\int_\eta^1f(\de_{N(\a,\eps)}(\a),\de_{N(\a,\eps)-1}(\a),\ldots,\de_{N(\a,\eps)-m}(\a))\frac{d\eps}{\eps}\to L_m(f)
$$
and
$$
\frac1{\ln(1/\eta)}\int_\eta^1(-1)^{N(\a,\eps)}f(\de_{N(\a,\eps)}(\a),\ldots,\de_{N(\a,\eps)-m}(\a))\frac{d\eps}{\eps}\to 0
$$
for a.e. $\a\in(0,1)$ as $\eta\to 0^+$.
\end{Thm}

\begin{proof}
The proof of the first limit is as in \cite{CagliotiFG2003}, and we just sketch it. Write
$$
\ba
\int_\eta^1f(\de_{N(\a,\eps)}(\a),\de_{N(\a,\eps)-1}(\a),\ldots,\de_{N(\a,\eps)-m}(\a))\frac{d\eps}{\eps}
\\
=
\sum_{n=1}^{N(\a,\eps)-1}
	\int_{d_n(\a)}^{d_{n-1}(\a)}f(\de_{N(\a,\eps)}(\a),\de_{N(\a,\eps)-1}(\a),\ldots,\de_{N(\a,\eps)-m}(\a))\frac{d\eps}{\eps}
\\
+
\int_{\eta}^{d_{N(\a,\eps)-1}(\a)}f(\de_{N(\a,\eps)}(\a),\de_{N(\a,\eps)-1}(\a),\ldots,\de_{N(\a,\eps)-m}(\a))\frac{d\eps}{\eps}
\ea
$$
Whenever $d_n(\a)\le\eps<d_{n-1}(\a)$, one has $N(\a,\eps)=n$ so that, for a.e. $\a\in(0,1)$, 
$$
\ba
\int_\eta^1f(\de_{N(\a,\eps)}(\a),\de_{N(\a,\eps)-1}(\a),\ldots,\de_{N(\a,\eps)-m}(\a))\frac{d\eps}{\eps}
\\
=
\sum_{n=1}^{N(\a,\eta)-1}\int_{d_n(\a)}^{d_{n-1}(\a)}f(\de_n(\a),\de_{n-1}(\a),\ldots,\de_{n-m}(\a))\frac{d\eps}{\eps}+O(1)\,.
\ea
$$ 
Substituting $\rho=\frac{d_n(\a)}{\eps}$ in each integral on the right hand side of the identity above, one has, for $n\ge m>1$
$$
\ba
\int_{d_n(\a)}^{d_{n-1}(\a)}f(\de_n(\a),\de_{n-1}(\a),\ldots,\de_{n-m}(\a))\frac{d\eps}{\eps}
\\
=
\int_{d_n(\a)/d_{n-1}(\a)}^1f\left(\rho,\frac{\rho}{T^{n-1}\a},\ldots,\frac{\rho}{\prod_{k=1}^mT^{n-m+k-1}\a}\right)\frac{d\rho}{\rho}
\\
=F_{m-1}(T^{n-m}\a)\,,
\ea
$$
with the notation
$$
F_{m-1}(\a)=\int_{T^{m-1}\a}^1f\left(\rho,\frac{\rho}{T^{m-1}\a},\ldots,\frac{\rho}{\prod_{k=1}^mT^k\a}\right)\frac{d\rho}{\rho}\,.
$$
Thus, for a.e. $\a\in(0,1)$, 
$$
\ba
\frac1{\ln(1/\eta)}\int_\eta^1f(\de_{N(\a,\eps)}(\a),\de_{N(\a,\eps)-1}(\a),\ldots,\de_{N(\a,\eps)-m}(\a))\frac{d\eps}{\eps}
\\
=\frac1{\ln(1/\eta)}\sum_{n=m}^{N(\a,\eta)-1}F_{m-1}(T^{n-m}\a)+O\left(\frac1{\ln(1/\eta)}\right)\,.
\ea
$$
We deduce from Birkhoff's ergodic theorem and Lemma \ref{L-EquivN} that
\begin{equation}
\label{Def-Lm}
\frac1{\ln(1/\eta)}\sum_{n=m}^{N(\a,\eta)-1}F_{m-1}(T^{n-m}\a)\to L_m(f):=\tfrac{12\ln 2}{\pi^2}\int_0^1F_{m-1}(x)dG(x)
\end{equation}
for a.e. $\a\in(0,1)$ as $\eta\to 0^+$, which establishes the first statement in the Theorem.

\smallskip
The proof of the second statement is fairly similar. We start from the identity 
$$
\ba
\int_\eta^1(-1)^{N(\a,\eps)}f(\de_{N(\a,\eps)}(\a),\de_{N(\a,\eps)-1}(\a),\ldots,\de_{N(\a,\eps)-m}(\a))\frac{d\eps}{\eps}&
\\
=
\sum_{n=1}^{N(\a,\eta)-1}(-1)^n\int_{d_n(\a)}^{d_{n-1}(\a)}f(\de_n(\a),\de_{n-1}(\a),\ldots,\de_{n-m}(\a))\frac{d\eps}{\eps}+O(1)&
\\
=\sum_{n=m}^{N(\a,\eta)-1}(-1)^nF_{m-1}(T^{n-m}\a)+O(1)&\,.
\ea
$$
Writing  
$$
\ba
\sum_{n=m}^{N(\a,\eta)-1}(-1)^nF_{m-1}(T^{n-m}\a)
\\
=\sum_{m/2\le k\le(N(\a,\eta)-1)/2}\left(F_{m-1}(T^{2k-m}\a)-F_{m-1}(T^{2k+1-m}\a)\right)+O(1)
\ea
$$
we deduce from Birkhoff's ergodic theorem (applied to $T^2$, which is ergodic since $T$ is mixing, instead of $T$) and Lemma \ref{L-EquivN} 
that, for a.e. $\a\in (0,1)$ and  in the limit as $\eta\to 0^+$,
$$
\ba
\frac1{\ln(1/\eta)}\sum_{m/2\le k\le(N(\a,\eta)-1)/2}\left(F_{m-1}(T^{2k-m}\a)-F_{m-1}(T^{2k+1-m}\a)\right)
\\
\to\tfrac{12\ln 2}{\pi^2}\int_0^1\left(F_{m-1}(x)-F_{m-1}(Tx)\right)dG(x)=0
\ea
$$
since the measure $dG$ is invariant under $T$. 

This entails the second statement in the theorem.
\end{proof}

\subsection{Application to 3-obstacle configurations}
\label{SS-Appli3ObstConfig}

Consider $\om\in\bS^1$ such that $0<\om_2<\om_1$ and $\om_2/\om_1\notin\bQ$, and let $r\in(0,\tfrac12)$.

The parameters $(A(\om,r),B(\om,r),Q(\om,r),\si(\om,r))$ defining the 3-obstacle configuration associated with the direction $\om$
and the obstacle radius $r$ are expressed in terms of the continued fraction expansion of $\om_2/\om_1$ in the following manner.

\begin{Prop}
\label{P-FlaABQsi}
For each $\om\in\bS^1$ such that $0<\om_2<\om_1$ and $\om_2/\om_1\notin\bQ$, and each $r\in(0,\tfrac12)$ one has
$$
\ba
A(\om,r)&=1-\frac{d_{N(\a,\eps)}(\a)}{\eps}
\\
B(\om,r)&=1-\frac{d_{N(\a,\eps)-1}(\a)}{\eps}-\left[\frac{\eps-d_{N(\a,\eps)-1}(\a)}{d_{N(\a,\eps)}(\a)}\right]\frac{d_{N(\a,\eps)}(\a)}{\eps}
\\
Q(\om,r)&=\eps q_{N(\a,\eps)}
\\
\si(\om,r)&=(-1)^{N(\a,\eps)}
\ea
$$
where
$$
\a=\frac{\om_2}{\om_1}\,,\quad\hbox{ and }\eps=\frac{2r}{\om_1}\,.
$$
By the definition of $N(\a,\eps)$, one has
$$
0\le A(\om,r)\,,\,\,B(\om,r)\,,\,\,Q(\om,r)\le 1\,.
$$
\end{Prop}

This is Proposition 2.2 on p. 205 in \cite{CagliotiFG2003}; see also Blank-Krikorian \cite{BlankKriko1993} on p. 726.

\smallskip
Our main result in this section is

\begin{Thm}
\label{T-ErgoABQsi}
Let $\bK=[0,1]^3\times\{\pm1\}$. For each $F\in C(\bK)$, there exists $\cL(F)\in\bR$ such that
$$
\frac1{\ln(1/\eta)}\int_\eta^{1/2}F(A(\om,r),B(\om,r),Q(\om,r),\si(\om,r))\frac{dr}{r}\to\cL(F)
$$
for a.e. $\om\in\bS^1$ such that $0<\om_2<\om_1$, in the limit as $\eta\to 0^+$.
\end{Thm}

\begin{proof}
First, observe that 
$$
F(A,B,Q,\sigma)=F_+(A,B,Q)+\sigma F_-(A,B,Q)
$$
with
$$
F_\pm(A,B,Q)=\tfrac12\left(F(A,B,Q,+1)\pm F(A,B,Q,-1)\right)\,.
$$
By Proposition \ref{P-FlaABQsi}, one has
$$
\ba
F_\pm(A(\om,r),B(\om,r),Q(\om,r))&
\\
=
F_\pm\left(1-\de_{N(\a,\eps)}(\a),1-\de_{N(\a,\eps)-1}(\a)-\left[\frac{1-\de_{N(\a,\eps)-1}(\a)}{\de_{N(\a,\eps)}(\a)}\right]\de_{N(\a,\eps)}(\a),\right.&
\\
\left.
\frac1{\de_{N(\a,\eps)-1}(\a)}d_{N(\a,\eps)-1}(\a)q_{N(\a,\eps)}(\a)\right)&\,.
\ea
$$
For each $m\ge 0$, we define
$$
\ba
f_{m,\pm}(\de_N,\ldots,\de_{N-m-1}):=&
\\
F_\pm\left(1\!-\!\de_N,1\!-\!\de_{N-1}\!-\!\left[\frac{1\!-\!\de_{N-1}}{\de_N}\right]\de_N,
	\frac1{\de_{N-1}}\!\sum_{j=(N-m)^+}^N(-1)^{N-j}\frac{\de_N\de_{N-1}}{\de_j\de_{j-1}}\right)&\,.
\ea
$$

Observe that
$$
\a T\a\le\tfrac12\hbox{ for each }\a\in(0,1)\setminus\bQ\,,
$$
so that, whenever $n>m+1$,
$$
\ba
\left|q_n(\a)d_{n-1}(\a)-\sum_{j=n-m}^n(-1)^{n-j}\frac{d_n(\a)d_{n-1}(\a)}{d_j(\a)d_{j-1}(\a)}\right|
	\le\frac{d_n(\a)d_{n-1}(\a)}{d_{n-m-1}(\a)d_{n-m-2}(\a)}&
\\
\le\prod_{k=n-m-1}^{n-1}T^k\a\cdot T^{k-1}\a\le 2^{-m-1}&\,.
\ea
$$
Likewise, whenever $n>m+1$ and $n>l+1$, 
\begin{equation}
\label{F-fm}
\left|\sum_{j=n-l}^n(-1)^{n-j}\frac{d_n(\a)d_{n-1}(\a)}{d_j(\a)d_{j-1}(\a)}
	-\sum_{j=n-m}^n(-1)^{n-j}\frac{d_n(\a)d_{n-1}(\a)}{d_j(\a)d_{j-1}(\a)}\right|\le 2^{-\min(l,m)}\,.
\end{equation}
Since $\de_{N(\a,\eps)-1}(\a)>1$ by definition of $N(\a,\eps)$, one has
\begin{equation}
\label{fm-fl}
\ba
|f_{m,\pm}(\de_{N(\a,\eps)}(\a),\ldots,\de_{N(\a,\eps)-m-1}(\a))-f_{l,\pm}(\de_{N(\a,\eps)}(\a),\ldots,\de_{N(\a,\eps)-l-1}(\a))|&
\\
\le\rho_\pm(2^{-\min(m,l)})&\,,
\ea
\end{equation}
where $\rho_\pm$ is a modulus of continuity for $F_\pm$ on the compact $[0,1]^3$.

The inequality (\ref{F-fm}) implies that
\begin{equation}
\label{Int-F-fm}
\ba
\Big|&\frac1{\ln(1/4\eta)}\int_\eta^{1/4}F_\pm(A(\om,r),B(\om,r),Q(\om,r))\frac{dr}{r}
\\
&-
\frac1{\ln(1/4\eta)}\int_{2\eta/\om_1}^{1/2\om_1}
	f_{m,\pm}(\de_{N(\a,\eps)}(\a),\ldots,\de_{N(\a,\eps)-m-1}(\a))\frac{d\eps}{\eps}\Big|\le\rho_\pm(2^{-m-1})
\ea
\end{equation}
upon setting $\a=\om_2/\om_1$. Moreover, by Theorem \ref{T-Ergo}, for each $\eta_*>0$ 
\begin{equation}
\label{LimInt-fm}
\ba
\frac1{\ln(\eta_*/\eta)}\int_\eta^{\eta_*}
f_{m,+}(\de_{N(\a,\eps)}(\a),\ldots,\de_{N(\a,\eps)-m-1}(\a))\frac{d\eps}{\eps}\to L_{m+1}(f_{m,+})
\\
\frac1{\ln(\eta_*/\eta)}\int_\eta^{\eta_*}(-1)^{N(\a,\eps)}f_{m,-}(\de_{N(\a,\eps)}(\a),\ldots,\de_{N(\a,\eps)-m-1}(\a))\frac{d\eps}{\eps}\to 0
\ea
\end{equation}
a.e. in $\a\in(0,1)$ as $\eta\to 0^+$, and the inequality (\ref{fm-fl}) implies that
$$
|L_{m+1}(f_{m,+})-L_{l+1}(f_{l,+})|\le\rho_\pm(2^{-\min(m,l)})\,,\quad l,m\ge 0\,.
$$
In other words, $(L_{m+1}(f_{m,+}))_{m\ge 0}$ is a Cauchy sequence. Therefore, there exists $L\in\bR$ such that
\begin{equation}
\label{LimLm}
L_{m+1}(f_{m,+})\to L\hbox{ as }n\to\infty\,.
\end{equation}

Putting together (\ref{Int-F-fm}), (\ref{LimInt-fm}) and (\ref{LimLm}), we first obtain
$$
\ba
L_{m+1}(f_{m,+})-2\rho(2^{-m-1})
\\
\le
\varliminf_{\eta\to 0^+}\frac1{\ln(1/4\eta)}\int_\eta^{1/4}F_+(A(\om,r),B(\om,r),Q(\om,r))\frac{dr}{r}
\\
\le
\varlimsup_{\eta\to 0^+}\frac1{\ln(1/4\eta)}\int_\eta^{1/4}F_+(A(\om,r),B(\om,r),Q(\om,r))\frac{dr}{r}
\\
\le
L_{m+1}(f_{m,+})+2\rho(2^{-m-1})\,,
\ea
$$
and
$$
\ba
-2\rho(2^{-m-1})
\\
\le
\varliminf_{\eta\to 0^+}\frac1{\ln(1/4\eta)}\int_\eta^{1/4}\si(\om,r)F_-(A(\om,r),B(\om,r),Q(\om,r))\frac{dr}{r}
\\
\le
\varlimsup_{\eta\to 0^+}\frac1{\ln(1/4\eta)}\int_\eta^{1/4}\si(\om,r)F_-(A(\om,r),B(\om,r),Q(\om,r))\frac{dr}{r}
\\
\le
2\rho(2^{-m-1})\,.
\ea
$$

Letting $m\to\infty$ in the above inequalities, we finally obtain that
$$
\ba
\frac1{\ln(1/4\eta)}\int_\eta^{1/4}F_+(A(\om,r),B(\om,r),Q(\om,r))\frac{dr}{r}\to L
\\
\frac1{\ln(1/4\eta)}\int_\eta^{1/4}\si(\om,r)F_-(A(\om,r),B(\om,r),Q(\om,r))\frac{dr}{r}\to 0
\ea
$$
a.e. in $\om\in\bS^1$ with $0<\om_2<\om_1$ as $\eta\to 0^+$.
\end{proof}

\smallskip
\noindent
\textbf{Amplification of Theorem \ref{T-ErgoABQsi}.} The proof given above shows that
$$
\mathcal{L}(F)=\mathcal{L}(F_+)
$$
for each $F\in C(\mathbf{K})$.


\section{Computation of the asymptotic distribution\\ of 3-obstacle configurations:
	\\ a proof of Proposition \ref{P-3ObstProba}}
\label{S-3ObstProba}

Having established the existence of the limit $\mathcal{L}(F)$ in Theorem \ref{T-ErgoABQsi}, we seek an explicit formula for it.

It would be most impractical to first compute $L_{m+1}(f_{m,+})$ --- with the notation of the proof of  Theorem \ref{T-ErgoABQsi} --- 
by its definition in formula (\ref{Def-Lm}), and then to pass to the limit as $m\to\infty$.

We shall instead use a different method based on Farey fractions and the asymptotic theory of Kloosterman's sums as in \cite{BocaZaharescu}.

\subsection{3-obstacle configurations and Farey fractions}
\label{SS-Farey}

For each integer $\cQ\ge 1$, consider the set of Farey fractions of order $\le\cQ$:
$$
\cF_\cQ:=\{\tfrac{p}q\,|\,1\le p\le q\le \cQ\,,\hbox{ g.c.d.}(p,q)=1\}\,.
$$
If $\g=\tfrac{p}q<\g'=\tfrac{p'}{q'}$ are two consecutive elements of $\cF_\cQ$, then
$$
q+q'>\cQ\,,\quad\hbox{ and }p'q-pq'=1\,.
$$
For each interval $I\subset[0,1]$, we denote 
$$
\cF_\cQ(I)=I\cap\cF_\cQ\,.
$$

The following lemma provides a (partial) dictionary between Farey and 
continued fractions.

\begin{Lem}
\label{L-Farey/Cont}
For each $0<\eps\,<1$, set $\cQ=[1/\eps]$.  Let $0<\a<1$ be irrational, and let $\g=\tfrac{p}q$ and $\g'=\tfrac{p'}{q'}$ be the two 
consecutive Farey fractions in $\cF_{\cQ}$ such that $\g<\a<\g'$. Then

\smallskip
\noindent
(i) if $\frac{p}q<\a<\frac{p'-\eps}{q'}$, then
$$
q_{N(\a,\eps)}(\a)=q\hbox{ and }d_{N(\a,\eps)}(\a)=q\a-p\,;
$$
(ii) if $\frac{p'-\eps}{q'}\le\a\le\frac{p+\eps}{q}$, then
$$
q_{N(\a,\eps)}(\a)=\min(q,q')
$$
while
$$
d_{N(\a,\eps)}(\a)=q\a-p\hbox{ if }q<q'\,,\hbox{Êand }d_{N(\a,\eps)}(\a)=p'-q'\a\hbox{ if }q'<q\,;
$$
(iii)  if $\frac{p+\eps}q<\a<\frac{p'}{q'}$, then
$$
q_{N(\a,\eps)}(\a)=q'\hbox{ and }d_{N(\a,\eps)}(\a)=p'-q'\a\,.
$$
\end{Lem}

\begin{proof}[Sketch of the proof]
According to Dirichlet's lemma, for each integer $\cQ\ge 1$, there exists an integer $\hat q$ such that $1\le\hat q\le\cQ$ and 
$\hbox{dist}(\hat q\a,\bZ)<\frac1{\cQ+1}$. If $\frac{p''}{q''}\in\cF_\cQ$ is different from $\frac{p}{q}$ and $\frac{p'}{q'}$, then $\hat q$
cannot be equal to $q''$. For if $\a<\frac{p'}{q'}<\frac{p''}{q''}$, then $p''-q''\a>\frac1{q''}(p''q'-p'q'')\ge\frac1{q''}\ge\frac{1}{\cQ+1}$.
Thus $\hat q$ is one of the two integers $q$ and $q'$. In the case (i) $p'-q'\a>\eps>\frac1{\cQ+1}$ so that $\hat q\not=q'$, hence
$\hat q=q$. Likewise, in the case (iii) $q\a-p>\eps>\frac{1}{\cQ+1}$ so that $\hat q=q'$. In the case (ii), one has 
$$
0\le\frac1{q'}(1-q\eps)=\frac1{q'}(qp'-pq'-q\eps)\le q\a-p\le\eps
$$
since $q\le\cQ\le\frac1\eps$. Likewise $0\le p'-q'\a\le\eps$, so that $\hat q$ is the smaller of $q$ and $q'$.
\end{proof}

\smallskip
In fact, the parameters $(A(\om,r),B(\om,r),Q(\om,r))$ can be computed in terms of Farey fractions, by a slight amplification of the 
proposition above. We recall that, for each $\om\in\bS^1$ such that $0<\om_2<\om_1$ and $\om_2/\om_1$ is irrational, one has
$$
Q(\om,r)=\eps q_{N(\a,\eps)}(\a)\,,\quad\hbox{ with }\a=\frac{\om_2}{\om_1}\hbox{ and }\eps=\frac{2r}{\om_1}\,.
$$
Under the same conditions on $\om$, we define
$$
D(\om,r):=\frac{d_{N(\a,\eps)}}{\eps}\,,\quad\hbox{ again with }\a=\frac{\om_2}{\om_1}\hbox{ and }\eps=\frac{2r}{\om_1}\,,
$$
and $Q'(\om,r)$ in the following manner. 

Let $\cQ=[1/\eps]$ with $\eps=\frac{2r}{\om_1}$, and let $\g=\frac{p}q$ and $\g'=\frac{p'}{q'}$ be the two consecutive Farey fractions 
in $\cF_{\cQ}$ such that $\g<\a<\g'$. Then

\smallskip
\noindent
(i) if $\frac{p}q<\a<\frac{p'-\eps}{q'}$, we set
$$
Q'(\om,r):=\eps q'\,;
$$
(ii) if $\frac{p'-\eps}{q'}\le\a\le\frac{p+\eps}{q}$, we set
$$
Q'(\om,r):=\eps\max(q,q')\,;
$$
(iii)  if $\frac{p+\eps}q<\a<\frac{p'}{q'}$, we set
$$
Q'(\om,r):=\eps q\,.
$$
With these definitions, the 3-obstacle configuration parameters are easily expressed in terms of Farey fractions, as follows.

\begin{Prop}
\label{P-FareyABQ}
Let $0<r<\tfrac14$ and $\om\in\bS^1$ be such that $0<\om_2<\om_1$ and $\om_2/\om_1$ is irrational.
Then
$$
A(\om,r)=1-D(\om,r)\,,
$$
while
$$
B(\om,r)=b(\om,r)-\left[\frac{b(\om,r)}{D(\om,r)}\right]D(\om,r)\,,
$$
with
$$
b(\om,r)=\frac{Q(\om,r)-1+Q'(\om,r)D(\om,r)}{Q(\om,r)}\,.
$$
\end{Prop}

\begin{proof}[Sketch of the proof]
We follow the discussion in Propositions 1 and 2 of \cite{BocaZaharescu}.  Consider the case $\frac{p}{q}<\a<\frac{p'-\eps}{q'}$, and
set $d=q\a-p>0$ and $d'=p'-q'\a>0$. According to Proposition 1 of \cite{BocaZaharescu}, $\eps B(\om,r)=\eps-(p_k-q_k\a)$, with the
notation $p_k=p'+kp$ et $q_k=q'+kq$ for $k\in\bN^*$ chosen so that
$$
\frac{p'+kp-\eps}{q'+kq}<\a<\frac{p'+(k-1)p-\eps}{q'+(k-1)q}\,.
$$
In other words, $k$ is chosen so that $d'-kd\le\eps<d'-(k-1)d$, where $d=q\a-p$ and $d'=p'-q'\a$. That is to say,
$$
-k=\left[\frac{\eps-d'}{d}\right]=\left[\frac{1-d'/\eps}{D}\right]\,,
$$
and
$$
B=1-\frac{d'-kd}{\eps}=1-\frac{d'}{\eps}+k\frac{d}{\eps}=b-\left[\frac{b}{D}\right]D
$$
with $b=1-\frac{d'}{\eps}$. Since $qp'-pq'=1$, one has $d'=(1-q'd)/q$, so that $b=\frac{Q-1-Q'D}{Q}$. The other cases are treated
similarly.
\end{proof}

\subsection{Asymptotic distribution of $(Q,Q',D)$}
\label{SS-DistQQ'D}

As a first step in computing $\cL(F)$, we establish the following

\begin{Lem}
\label{L-DistQQ'D}
Let $f\in C([0,1]^3)$ and $J=[\a_-,\a_+]\subset(0,1)$. Then
$$
\ba
\int_{\bS^1_+}f(Q(\om,\tfrac12\eps\om_1),Q'(\om,\tfrac12\eps\om_1),D(\om,\tfrac12\eps\om_1))
	\indc_{\om_2/\om_1\in J}\frac{d\om}{\om_1^2}
\\
\to|J|\int_{[0,1]^3}f(Q,Q',D)d\l(Q,Q',D)
\ea
$$
as $\eps\to 0^+$, where $\l$ is the probability measure on $[0,1]^3$ given by
$$
\ba
d\lambda(Q,Q',D)=
\\
\tfrac{12}{\pi^2}
\left(\indc_{0<Q,Q'<1}\indc_{Q+Q'>1}\indc_{0<D<\frac{1-Q}{Q'}}+\indc_{Q<Q'}\indc_{\frac{1-Q}{Q'}<D<1}\right)\frac{dQdQ'dD}{Q}
\ea
$$
\end{Lem}

The proof of this lemma is based on the arguments involving Kloosterman's sums to be found in \cite{BocaZaharescu}.

\begin{proof}
For $0<\eps<1$, set $\cQ=[1/\eps]$. Observe that
$$
\ba
\left|Q(\om,\tfrac12\eps\om_1)-\frac1{\eps\cQ}Q(\om,\tfrac12\eps\om_1)\right|
	&\le\frac1{\cQ+1}\frac{Q(\om,\tfrac12\eps\om_1)}{\eps\cQ}\le\frac1{\cQ+1}\,,
\\
\left|Q'(\om,\tfrac12\eps\om_1)-\frac1{\eps\cQ}Q'(\om,\tfrac12\eps\om_1)\right|
	&\le\frac1{\cQ+1}\frac{Q'(\om,\tfrac12\eps\om_1)}{\eps\cQ}\le\frac1{\cQ+1}\,,
\ea
$$
since $Q(\om,\tfrac12\eps\om_1)\le\eps\cQ$ and $Q'(\om,\tfrac12\eps\om_1)\le\eps\cQ$ while
$$
|D(\om,\tfrac12\eps\om_1)-\eps\cQ D(\om,\tfrac12\eps\om_1)|\le\eps\left|\frac1\eps-\cQ\right|\le\eps\le\frac1{\cQ}
$$
since $D(\om,\tfrac12\eps\om_1)\le 1$. 

Since $f$ is continuous on the compact set $[0,1]^3$, it is uniformly continuous; let $\rho$ be a modulus of continuity for $f$. Then
\begin{equation}
\label{Diff-eps-Q}
\ba
\Bigg|\int_{\bS^1_+}
	f(Q(\om,\tfrac12\eps\om_1),Q'(\om,\tfrac12\eps\om_1),D(\om,\tfrac12\eps\om_1))\frac{d\om}{\om_1^2}
\hbox{\hskip3cm}
\\
-
\int_{\bS^1_+}f\left(\frac1{\eps\cQ}Q(\om,\tfrac12\eps\om_1),
	\frac1{\eps\cQ}Q'(\om,\tfrac12\eps\om_1),\eps\cQ D(\om,\tfrac12\eps\om_1)\right)\frac{d\om}{\om_1^2}\Bigg|
\\
\le 3\rho(1/\cQ)\to 0
\ea
\end{equation}
as $\eps\to 0^+$.

Henceforth, we seek to prove that
\begin{equation}
\label{Lim-int-f-Q}
\ba
\int_{\bS^1_+}f\left(\frac1{\eps\cQ}Q(\om,\tfrac12\eps\om_1),
\frac1{\eps\cQ}Q'(\om,\tfrac12\eps\om_1),\eps\cQ D(\om,\tfrac12\eps\om_1)\right)\indc_{\om_2/\om_1\in J}\frac{d\om}{\om_1^2}
\\
\to|J|\int_{[0,1]^3}f(Q,Q',D)d\lambda(Q,Q',D)
\ea
\end{equation}
as $\eps\to 0^+$. Without loss of generality, by an obvious density argument we restrict our attention to the case where the test
function $f$ is of the form
$$
f(Q,Q',D)=g(Q,Q')h(D)\,,
$$
with $g\in C([0,1]^2)$ and $h\in C([0,1])$. Then
\begin{equation}
\label{Decomp-Int-f}
\ba
\int_{\bS^1_+}g\left(\frac1{\eps\cQ}Q(\om,\tfrac12\eps\om_1),\frac1{\eps\cQ}Q'(\om,\tfrac12\eps\om_1)\right)
	h\left(\eps\cQ D(\om,\tfrac12\eps\om_1)\right)\indc_{\om_2/\om_1\in J}\frac{d\om}{\om_1^2}
\\
=\sum_{p/q\in\cF_{\cQ}(J)}\int_{\tfrac{p}q}^{\tfrac{p'-\eps}{q'}}g\left(\frac{q}{\cQ},\frac{q'}{\cQ}\right)h(\cQ(q\a-p))d\a
\\
+
\sum_{p/q\in\cF_{\cQ}(J)}\int_{\tfrac{p+\eps}q}^{\tfrac{p'}{q'}}g\left(\frac{q'}{\cQ},\frac{q}{\cQ}\right)h(\cQ(p'-q'\a))d\a
\\
+\sum_{p/q\in\cF_{\cQ}(J)}\int_{\tfrac{p'-\eps}{q'}}^{\tfrac{p+\eps}q}\indc_{q<q'}g\left(\frac{q}{\cQ},\frac{q'}{\cQ}\right)h(\cQ(q\a-p))d\a
\\
+\sum_{p/q\in\cF_{\cQ}(J)}\int_{\tfrac{p'-\eps}{q'}}^{\tfrac{p+\eps}q}\indc_{q'<q}g\left(\frac{q'}{\cQ},\frac{q}{\cQ}\right)h(\cQ(p'-q'\a))d\a
\\
+\|g\|_{L^\infty}\|h\|_{L^\infty}O\left(\sup_{p/q\in\cF_{\cQ}(J)}\left|\frac{p'}{q'}-\frac{p}q\right|
	+\sup_{p'/q'\in\cF_{\cQ}(J)}\left|\frac{p'}{q'}-\frac{p}q\right|\right)
\\
=I+II+III+IV+O(1/\cQ)
\ea
\end{equation}
where $\frac{p}q<\frac{p'}{q'}$ are consecutive elements of $\cF_{\cQ}$. Notice that the $O$ term above is the contribution of the
endpoints of $J$ not being elements of $\cF_Q$. 

Define
$$
H(z)=\int_0^zh(y)dy\,,\quad\hbox{ for each }z\in\bR\,.
$$
Thus
$$
\ba
I&=\sum_{p/q\in\cF_{\cQ}(J)}g\left(\frac{q}{\cQ},\frac{q'}{\cQ}\right)\frac1{\cQ q}H\left(\cQ\frac{1-\eps q}{q'}\right)
\\
II&=\sum_{p/q\in\cF_{\cQ}(J)}g\left(\frac{q'}{\cQ},\frac{q}{\cQ}\right)\frac1{\cQ q'}H\left(\cQ\frac{1-\eps q'}{q}\right)
\\
III&=\sum_{p/q\in\cF_{\cQ}(J)}\indc_{q<q'}g\left(\frac{q}{\cQ},\frac{q'}{\cQ}\right)
						\frac1{\cQ q}\left(H(\eps\cQ)-H\left(\cQ\frac{1-\eps q}{q'}\right)\right)
\\
IV&=\sum_{p/q\in\cF_{\cQ}(J)}\indc_{q'<q}g\left(\frac{q'}{\cQ},\frac{q}{\cQ}\right)
						\frac1{\cQ q'}\left(H(\eps\cQ)-H\left(\cQ\frac{1-\eps q'}{q}\right)\right)
\ea
$$
One has
$$
\ba
\left|H\left(\cQ\frac{1-\eps q}{q'}\right)-H\left(\cQ\frac{1-q/\cQ}{q'}\right)\right|\le\|h\|_{L^\infty}\cQ\frac{q}{q'}\left|\frac1{\cQ}-\eps\right|
\\
\le\|h\|_{L^\infty}\cQ\frac{q}{q'}\frac1{\cQ(\cQ+1)}\,;
\ea
$$
likewise
$$
\left|H\left(\cQ\frac{1-\eps q'}{q}\right)-H\left(\cQ\frac{1-q'/\cQ}{q}\right)\right|\le\|h\|_{L^\infty}\cQ\frac{q'}{q}\frac1{\cQ(\cQ+1)}\,,
$$
while
$$
\ba
|H(1)-H(\eps\cQ)|\le\|h\|_{L^\infty}|1-\eps\cQ|=\|h\|_{L^\infty}\eps\left|\frac1{\eps}-\cQ\right|\le\|h\|_{L^\infty}\eps\,.
\ea
$$
Hence
$$
I=\sum_{p/q\in\cF_{\cQ}(J)}g\left(\frac{q}{\cQ},\frac{q'}{\cQ}\right)\frac1{\cQ q}H\left(\frac{1-q/\cQ}{q'/\cQ}\right)+R_I
$$
with
$$
|R_I|\le\|g\|_{L^\infty}\|h\|_{L^\infty}\sum_{p/q\in\cF_{\cQ}(J)}\frac1{q'\cQ(\cQ+1)}\,.
$$
Applying Lemmas 1 and 2 in \cite{BocaZaharescu} shows that
$$
\ba
\sum_{p/q\in\cF_{\cQ}(J)}g\left(\frac{q}{\cQ},\frac{q'}{\cQ}\right)\frac1{\cQ q}H\left(\frac{1-q/\cQ}{q'/\cQ}\right)
\\
\to|J|\cdot\tfrac{6}{\pi^2}\iint_{(0,1)^2}g(Q,Q')H\left(\frac{1-Q}{Q'}\right)\indc_{Q+Q'>1}\frac{dQdQ'}{Q}\hbox{ as }\cQ\to\infty\,,
\ea
$$
while
$$
|R_I|\lesssim\|g\|_{L^\infty}\|h\|_{L^\infty}\frac{1}{\cQ}\cdot|J|\cdot\tfrac{6}{\pi^2}\iint_{(0,1)^2}\indc_{Q+Q'>1}\frac{dQdQ'}{Q'}=O(1/\cQ)\,,
$$
so that
\begin{equation}
\label{Lim-I}
I\to|J|\cdot\tfrac{6}{\pi^2}\iint_{(0,1)^2}g(Q,Q')H\left(\tfrac{1-Q}{Q'}\right)\indc_{Q+Q'>1}\tfrac{dQdQ'}{Q}\hbox{ as }\cQ\to\infty\,.
\end{equation}

By a similar argument, one shows that
\begin{equation}
\label{Lim-II-III-IV}
\ba
II&\to|J|\cdot\tfrac{6}{\pi^2}\iint_{(0,1)^2}g(Q',Q)H\left(\tfrac{1-Q'}{Q}\right)\indc_{Q+Q'>1}\tfrac{dQdQ'}{Q'}\,,
\\
III&\to|J|\cdot\tfrac{6}{\pi^2}\iint_{(0,1)^2}g(Q,Q')\left(H(1)-H\left(\tfrac{1-Q}{Q'}\right)\right)\indc_{Q<Q'}\indc_{Q+Q'>1}\tfrac{dQdQ'}{Q}\,,
\\
IV&\to|J|\cdot\tfrac{6}{\pi^2}\iint_{(0,1)^2}g(Q',Q)\left(H(1)-H\left(\tfrac{1-Q'}{Q}\right)\right)\indc_{Q'<Q}\indc_{Q+Q'>1}\tfrac{dQdQ'}{Q'}\,,
\ea
\end{equation}
as $\cQ\to\infty$.

Substituting the limits (\ref{Lim-I})-(\ref{Lim-II-III-IV}) in (\ref{Decomp-Int-f}) shows that
$$
\ba
\int_{\bS^1_+}
g\left(\frac1{\eps\cQ}Q(\om,\tfrac12\eps\om_1),\frac1{\eps\cQ}Q'(\om,\tfrac12\eps\om_1)\right)
	h\left(\eps\cQ D(\om,\tfrac12\eps\om_1)\right)\indc_{\om_2/\om_1\in J}\frac{d\om}{\om_1^2}
\\
\to|J|\cdot\tfrac{12}{\pi^2}\iiint_{(0,1)^3}g(Q,Q')h(D)\indc_{Q+Q'>1}\indc_{0<D<\frac{1-Q}{Q'}}\tfrac{dQdQ'dD}{Q}
\\
+
|J|\cdot\tfrac{12}{\pi^2}\iiint_{(0,1)^3}g(Q,Q')h(D)\indc_{Q<Q'}\indc_{Q+Q'>1}\indc_{\frac{1-Q}{Q'}<D<1}\tfrac{dQdQ'dD}{Q}	
\ea
$$
as $\eps\to 0$, which establishes (\ref{Lim-int-f-Q}). 

On account of (\ref{Diff-eps-Q}), this proves the convergence announced in Lemma \ref{L-DistQQ'D}.
\end{proof}

\subsection{Asymptotic distribution of $(A,B,Q)$}
\label{SS-DistABQ}

Next we compute the image the probability measure $\l$ under the map $(Q,Q',D)\mapsto(A,B,Q)$ defined in Proposition 
\ref{P-FareyABQ}. In other words:

\begin{Lem}
\label{L-DistABQ}
Let $g\in C([0,1]^3)$. Then
$$
\ba
\int_{\bS^1_+}g(A(\om,\tfrac12\eps\om_1),B(\om,\tfrac12\eps\om_1),Q(\om,\tfrac12\eps\om_1))\frac{d\om}{\om_1^2}
\\
\to\int_{[0,1]^3}g(A,B,Q)d\nu(A,B,Q)
\ea
$$
as $\eps\to 0^+$, where $\nu$ is the probability measure on $[0,1]^3$ given by
$$
d\nu(A,B,Q)\!=\!{\frac{12}{\pi^2}}\indc_{0<A<1}\indc_{0<B<1-A}\indc_{0<Q<\frac1{2-A-B}}\frac{dAdBdQ}{1-A}\,.
$$
\end{Lem}

\begin{proof}
We first compute the image $\Phi_*\l$ of the probability measure $\l$ in Lemma \ref{L-DistQQ'D} under the map 
$$
\ba
\Phi:\,&[0,1]^3\to[0,1]\times\bR\times[0,1]
\\
&(Q,Q',D)\mapsto(A,b,Q)=(1-D,\tfrac{Q-1+Q'D}{Q},Q)
\ea
$$
The Jacobian of $\Phi$ is 
$$
\hbox{det}\nabla\Phi(Q,Q',D)
=\left|\begin{matrix} 0 &\frac{1-Q'D}{Q^2} &1\\ {} &{} &\\
0 &\frac{D}Q &0\\ {} &{} &{}\\ -1 &\frac{Q'}{Q} &0\end{matrix}\right|
=\frac{D}Q=\frac{1-A}Q\,,
$$
A straightforward computation shows that
$$
\ba
d\Phi_*\lambda(A,b,Q)=
\\
\tfrac{12}{\pi^2}\indc_{0<A,Q<1}\indc_{1-\frac{1}{Q}<b<1-\frac{A}{Q}}\left(\indc_{A-\frac{A}{Q}<b<0}
	+\indc_{b>0}\indc_{b>2-A-\frac{1}{Q}}\right)\frac{dAdbdQ}{1-A}\,.
\ea
$$
This expression can be put in the form
\begin{equation}
\label{Phi*lambda}
d\Phi_*\lambda(A,b,Q)=
	\tfrac{12}{\pi^2}\indc_{0<A,Q<1}(\indc_{b\in\Lambda_1(A,Q)}+\indc_{b\in\Lambda_2(A,Q)})\frac{dAdbdQ}{1-A}
\end{equation}
where
$$
\ba
\Lambda_1(A,Q)&=(A-\tfrac{A}Q,(1-\tfrac{A}Q)\wedge 0)
\\
\Lambda_2(A,Q)&=((2-A-\tfrac1Q)\vee 0,(1-\tfrac{A}Q))
\ea
$$
The probability measure $\nu$ is the image of $\Phi_*\lambda$ under the map
$$
\ba
\Psi:\,&[0,1]\times\bR\times[0,1]\to[0,1]^3
\\
&(A,b,Q)\mapsto(A,B,Q)=(A,b-[\tfrac{b}{1-A}](1-A),Q)
\ea
$$
In other words,
\begin{equation}
\label{Fla1-nu}
\ba
d\nu(A,B,Q)&=\tfrac{12}{\pi^2}\indc_{0<A,Q<1}\indc_{0<B<1-A}\frac{dAdbdQ}{1-A}
\\
&\cdot\sum_{n\in\bZ}\left(\indc_{\Lambda_1(A,Q)}(B+n(1-A))+\indc_{\Lambda_2(A,Q)}(B+n(1-A))\right)
\\
&=\tfrac{12}{\pi^2}M(A,B,Q)\indc_{0<A,Q<1}\indc_{0<B<1-A}\frac{dAdbdQ}{1-A}
\ea
\end{equation}
where
\begin{equation}
\label{Fla1-M}
\ba
M(A,B,Q)&=\#\{n\in\bZ\,|\,B+n(1-A)\in\Lambda_1(A,Q)\}\,,
\\
&+\#\{n\in\bZ\,|\,B+n(1-A)\in\Lambda_2(A,Q)\}\,.
\ea
\end{equation}
Whenever $u,v\notin\bZ$, one has
$$
\#\{m\in\bZ\,|\,m\in(u,v)\}=([v]-[u])_+\,.
$$

Hence, for a.e. $A,Q\in[0,1]$
$$
\ba
\#\{n\in\bZ\,|\,B+n(1-A)\in\Lambda_1(A,Q)\}
\\
=\left[\frac{(1-A/Q)\wedge 0-B}{1-A}\right]-\left[\frac{A-A/Q-B}{1-A}\right]
\ea
$$
--- since
$$
\frac{(1-A/Q)\wedge 0-B}{1-A}-\frac{A-A/Q-B}{1-A}=\frac{(1-A)\wedge(A/Q-A)}{1-A}\ge 0\,.
$$
For a.e. $A,Q\in[0,1]$, if $Q\le A$
$$
\ba
\left[\frac{(1-A/Q)\wedge 0-B}{1-A}\right]-\left[\frac{A-A/Q-B}{1-A}\right]
\\
=
\left[\frac{1-A/Q-B}{1-A}\right]-\left[\frac{A-A/Q-B}{1-A}\right]
\\
=\left[\frac{A-A/Q-B}{1-A}+1\right]-\left[\frac{A-A/Q-B}{1-A}\right]=1\,.
\ea
$$

Likewise, for a.e. $A,Q\in[0,1]$
$$
\ba\#\{n\in\bZ\,|\,B+n(1-A)\in\Lambda_2(A,Q)\}
\\
=\indc_{Q>A}\left(\left[\frac{1-A/Q-B}{1-A}\right]-\left[\frac{(2-A-1/Q)\vee 0-B}{1-A}\right]\right)
\ea
$$
--- observe that, whenever $Q\le A$, one has $2-A-1/Q\le 2-A-1/A\le 0$, so that
$$
\frac{1-A/Q-B}{1-A}-\frac{(2-A-1/Q)\vee 0-B}{1-A}=\frac{1-A/Q}{1-A}\le 0\,,
$$
while, if $Q>A$,
$$
\ba
\frac{1-A/Q-B}{1-A}-\frac{(2-A-1/Q)\vee 0-B}{1-A}
\\
=
\frac{(1-A/Q)\wedge(A-1)(1-1/Q)}{1-A}\ge 0\,.
\ea
$$

Therefore
$$
\ba
\#\{n\in\bZ\,|\,B+n(1-A)\in\Lambda_1(A,Q)\}
\\
+
\#\{n\in\bZ\,|\,B+n(1-A)\in\Lambda_2(A,Q)\}
\\
=\indc_{Q\le A}+\indc_{Q>A}\Bigg(\left[\frac{(1-A/Q)\wedge 0-B}{1-A}\right]-\left[\frac{A-A/Q-B}{1-A}\right]
\\
+\left[\frac{1-A/Q-B}{1-A}\right]-\left[\frac{(2-A-1/Q)\vee 0-B}{1-A}\right]\Bigg)
\\
=
\indc_{Q\le A}+\indc_{Q>A}\Bigg(\left[\frac{(1-A/Q)\wedge 0-B}{1-A}\right]+1
\\
-\left[\frac{(2-A-1/Q)\vee 0-B}{1-A}\right]\Bigg)
\\
=1+\indc_{Q>A}\Bigg(\left[\frac{-B}{1-A}\right]-\left[\frac{(2-A-1/Q)\vee 0-B}{1-A}\right]\Bigg)\,.
\ea
$$
Now, if $Q\le\frac1{2-A}$,
$$
\left[\frac{(2-A-1/Q)\vee 0-B}{1-A}\right]=\left[\frac{-B}{1-A}\right] 
$$
while $Q>A$ if $Q>\frac1{2-A}$ since $A+1/A>2$, so that
$$
\ba
M(A,B,Q)&=\#\{n\in\bZ\,|\,B+n(1-A)\in\Lambda_1(A,Q)\}
\\
&+\#\{n\in\bZ\,|\,B+n(1-A)\in\Lambda_2(A,Q)\}
\\
&=1+\indc_{Q>\frac1{2-A}}\Bigg(\left[\frac{-B}{1-A}\right]-\left[\frac{2-A-1/Q-B}{1-A}\right]\Bigg)\,.
\ea
$$
If $\frac1{2-A}<Q\le 1$ and $0<B<1-A$, one has
$$
-(1-A)<2-A-1/Q-B<(1-A)
$$
so that
$$
\ba
\left[\frac{-B}{1-A}\right]&=-1\,,&&
\\
\left[\frac{2-A-1/Q-B}{1-A}\right]&=0\,,&&\hbox{ if }2-A-B\ge 1/Q\,,
\\
\left[\frac{2-A-1/Q-B}{1-A}\right]&=-1\,,&&\hbox{ if }2-A-B<1/Q\,.
\ea
$$
Therefore
\begin{equation}
\label{Fla2-M}
\ba
M(A,B,Q)=1+\indc_{Q>\frac1{2-A}}\Bigg(\left[\frac{-B}{1-A}\right]-\left[\frac{2-A-1/Q-B}{1-A}\right]\Bigg)
\\
=1-\indc_{Q>\frac1{2-A}}\indc_{2-A-B\ge 1/Q}=\indc_{Q<\frac{1}{2-A-B}}
\ea
\end{equation}
whenever $B\not=0$, and this establishes the formula for $\nu$ in the lemma.
\end{proof}

\subsection{Computation of $\mathcal{L}(F)$}
\label{SS-ComputL(F)}

With the help of Lemma \ref{L-DistABQ}, we finally compute $\cL(F)$.

\begin{proof}[Proof of Proposition \ref{P-3ObstProba}]
Let $F\in C(\bK)$, and set 
$$
F_+(A,B,Q)=\tfrac12\left(F(A,B,Q,+1)+F(A,B,Q,-1)\right)\,;
$$
obviously $F_+\in C([0,1]^3)$. 

By Lemma \ref{L-DistABQ}
$$
\ba
\frac1{\ln(\eta_*/\eta)}\int_\eta^{\eta_*}\left(\int_{\bS^1_+}
	F_+(A(\om,\tfrac12\eps\om_1),B(\om,\tfrac12\eps\om_1),Q(\om,\tfrac12\eps\om_1))\frac{d\om}{\om_1^2}\right)\frac{d\eps}{\eps}
\\
\to\tfrac{\pi}4\int_{[0,1]^3}F_+(A,B,Q)d\nu(A,B,Q)
\ea
$$
as $\eta\to 0^+$, for each $\eta_*>0$.

On the other hand, by Fubini's theorem
$$
\ba
\frac1{\ln(\eta_*/\eta)}\int_\eta^{\eta_*}\left(\int_{\bS^1_+}
	F_+(A(\om,\tfrac12\eps\om_1),B(\om,\tfrac12\eps\om_1),Q(\om,\tfrac12\eps\om_1))\frac{d\om}{\om_1^2}\right)	\frac{d\eps}{\eps}
\\
=\int_{\bS^1_+}\left(\frac1{\ln(\eta_*/\eta)}\int_\eta^{\eta_*}
	F_+(A(\om,\tfrac12\eps\om_1),B(\om,\tfrac12\eps\om_1),Q(\om,\tfrac12\eps\om_1))\frac{d\eps}{\eps}\right)\frac{d\om}{\om_1^2}
\\
=\int_{\bS^1_+}\left(\frac1{\ln(\eta_*/\eta)}\int_{\om_1\eta/2}^{\om_1\eta_*/2}
	F_+(A(\om,r),B(\om,r),Q(\om,r))\frac{dr}{r}\right)	\frac{d\om}{\om_1^2}\,.
\ea
$$
By Theorem \ref{T-ErgoABQsi}, the inner integral on the right hand side of the equality above converges a.e. in $\om\in\bS^1_+$ to 
$\cL(F_+)$. By dominated convergence, we therefore conclude that
$$
\cL(F_+)=\int_{[0,1]^3}F_+(A,B,Q)d\nu(A,B,Q)\,.
$$
Finally, according to the amplification of Theorem \ref{T-ErgoABQsi}
$$
\cL(F)=\int_{[0,1]^3}F_+(A,B,Q)d\nu(A,B,Q)
$$
or, in other words,
$$
\cL(F)=\int_{[0,1]^3\times\{\pm1\}}F(A,B,Q,\si)d\mu(A,B,Q,\si)
$$
where
$$
d\mu(A,B,Q,\si)=\tfrac12d\nu(A,B,Q)\otimes(\de_{\si=+1}+\de_{\si=-1})\,.
$$
\end{proof}


\section{The transition probability: a proof of Theorem \ref{T-LimTr}}
\label{S-TransitProba}

\subsection{Computation of $P(S,h|h')$}
\label{SS-ComputP}

We first establish that the image of the probability $\nu$ defined in Lemma \ref{L-DistABQ} under the map 
$$
[0,1]^3\ni(A,B,Q)\mapsto\bT_{A,B,Q,+1}(h')\in\bR_+\times[-1,1]
$$
is of the form $P_+(S,h|h')dSdh$, with a probability density $P_+$ which we compute in the present section. Let 
$f\in C_c(\bR_+\times[-1,1])$; the identity
$$
\int_{\bR_+\times[-1,1]}f(S,h)P_+(S,h|h')dSdh=\int_{[0,1]^3}f(\bT_{A,B,Q,+1}(h'))d\nu(A,B,Q)
$$
defining $P_+(S,h|h')$ is recast as 
$$
\ba
\int_{\bR_+\times[-1,1]}f(S,h)P_+(S,h|h')dSdh
\\
=\int_{[0,1]^3}f(Q,h'-2(1-A))\indc_{1-2A<h'<1}d\nu(A,B,Q)
\\
+\int_{[0,1]^3}f\left(\tfrac{1-Q(1-B)}{1-A}, h'+2(1-B)\right)\indc_{-1<h'<-1+2B}d\nu(A,B,Q)
\\
+\int_{[0,1]^3}f\left(\tfrac{1-Q(B-A)}{1-A}, h'+2(A-B)\right)\indc_{-1+2B<h'<1-2A}d\nu(A,B,Q)
\\
=I+II+III\,.
\ea
$$

In I, we first integrate in $B$, since the argument of $f$ does not involve $B$. Observing that $0\vee(2-A-\tfrac1Q)<B<1-A$, we have
to distinguish two cases, namely $0<Q<\tfrac1{2-A}$ and $\tfrac1{2-A}<Q<1$, which leads to 
$$
I=\tfrac{12}{\pi^2}
	\iint\!f(Q,h'-2(1-A))\indc_{(1-h')/2<A<1}\left(\indc_{0<Q<\frac1{2-A}}\!+\!\tfrac{1-Q}{Q(1-A)}\indc_{\frac1{2-A}<Q<1}\right)dQdA\,.
$$
In the right-hand side of the equality above, we set $S=Q$ and $h=h'-2(1-A)$, to find
$$
I=\tfrac6{\pi^2}\iint f(S,h)\indc_{-1<h<h'<1}\left(\indc_{0<S<\frac1{1+(h'-h)/2}}+\tfrac{2(1-S)}{S(h'-h)}\indc_{\frac1{1+\eta}<S<1}\right)dSdh\,.
$$
With the notation $\eta=\tfrac12|h-h'|$, we see that $I$ can be put in the form
$$
\ba
I&=\tfrac6{\pi^2}\iint f(S,h)\indc_{-1<h<h'<1}\left(S\eta\indc_{0<S<\frac1{1+\eta}}
	+(1-S)\indc_{\frac1{1+(h'-h)/2}<S<1}\right)\frac{dSdh}{S\eta}
\\
&=\tfrac6{\pi^2}\iint f(S,h)\indc_{-1<h<h'<1}\indc_{0<S<1}(S\eta)\wedge(1-S)\frac{dSdh}{S\eta}\,.
\ea
$$

As for II, given $A$, we first make the substitution 
$$
(S,h)=\left(\tfrac{1-Q(1-B)}{1-A}, h'+2(1-B)\right)
$$
whose Jacobian is
$$
\frac{\d(S,h)}{\d(Q,B)}=\left|\begin{matrix}-\tfrac{1-B}{1-A}\,\,\,&0\\ \tfrac{Q}{1-A}&-2\end{matrix}\right|
	=\frac{2(1-B)}{1-A}=\frac{h-h'}{1-A}\,.
$$
Hence
$$
\ba
II=\tfrac{12}{\pi^2}\int_0^1\left(\iint f(S,h)\indc_{\frac1{1-A+(h-h')/2}<S<\frac1{1-A}}\indc_{A<\frac{h-h'}2<1}\tfrac{dSdh}{|h-h'|}\right)dA
\\
=\tfrac{12}{\pi^2}\iint f(S,h)\left(\int_0^1\indc_{\frac1{1-A+(h-h')/2}<S<\frac1{1-A}}\indc_{A<\frac{h-h'}2<1}dA\right)\frac{dSdh}{|h-h'|}\,.
\ea
$$
The inner integral is recast as
$$
\ba
\int_0^1\indc_{\frac1{1-A+(h-h')/2}<S<\frac1{1-A}}\indc_{A<\frac{h-h'}2<1}dA
		\qquad\qquad\qquad\qquad\qquad\qquad\qquad\qquad
\\
=
\indc_{S>0}\indc_{-1<h'<h<1}\int_0^1\indc_{1-1/S<A<1-1/S+(h-h')/2}\indc_{A<\frac{h-h'}2}dA
\\
=
\indc_{S>0}\indc_{-1<h'<h<1}\indc_{S<1}(1+\tfrac{h-h'}2-\tfrac1S)_+
\\
+
\indc_{S>0}\indc_{-1<h'<h<1}\indc_{S>1}(\tfrac{h-h'}2+\tfrac1S-1)_+
\\
=\indc_{S>0}\indc_{-1<h'<h<1}(\tfrac{h-h'}2-|\tfrac1S-1|)_+\,,
\ea
$$
so that, eventually, we find
$$
\ba
II&=\tfrac{12}{\pi^2}\iint f(S,h)\indc_{S>0}\indc_{-1<h'<h<1}(\tfrac{h-h'}2-|\tfrac1S-1|)_+\frac{dSdh}{h-h'}
\\
&=\tfrac{6}{\pi^2}\iint f(S,h)\indc_{S>0}\indc_{-1<h'<h<1}(S\eta-|1-S|)_+\frac{dSdh}{S\eta}
\ea
$$

Now for III. Here we make the substitution
$$
(S,h)=(\tfrac{1-Q(A-B)}{1-A},h'+2(A-B))
$$
whose Jacobian is
$$
\frac{\d(S,h)}{\d(Q,B)}=\left|\begin{matrix}-\tfrac{A-B}{1-A}\,\,\,&0\\ \tfrac{Q}{1-A}&-2\end{matrix}\right|
	=\frac{2(A-B)}{1-A}=\frac{h-h'}{1-A}\,.
$$
This Jacobian vanishes for $A=B$; therefore, we further decompose
$$
III=III_1+III_2
$$
with the notations
$$
\ba
III_1&=\int_{[0,1]^3}f\left(\tfrac{1-Q(B-A)}{1-A}, h'+2(A-B)\right)\indc_{-1+2B<h'<1-2A}\indc_{A<B}d\nu(A,B,Q)\,,
\\
III_2&=\int_{[0,1]^3}f\left(\tfrac{1-Q(B-A)}{1-A}, h'+2(A-B)\right)\indc_{-1+2B<h'<1-2A}\indc_{A>B}d\nu(A,B,Q)\,. 
\ea
$$

We begin with $III_1$. With the substitution above, one has
$$
\ba
III_1=\tfrac{12}{\pi^2}\int_0^1\Bigg(\iint f(S,h)\indc_{2A-1<(h-h')/2<A}\indc_{h'<1-2A}\indc_{-1+2A<h}
\\
\indc_{1/(1-A)<S<1/(1-A+(h-h')/4)}\indc_{-1<h<h'<1}\tfrac{dSdh}{|h-h'|}\Bigg)dA
\\
=\tfrac{12}{\pi^2}\iint f(S,h)\Bigg(\int_0^1\indc_{2A-1<(h-h')/2<A}\indc_{h'<1-2A}\indc_{-1+2A<h}
\\
\indc_{1/(1-A)<S<1/(1-A+(h-h')/4)}dA\Bigg)\indc_{-1<h<h'<1}\tfrac{dSdh}{|h-h'|}\,.
\ea
$$
The inner integral is 
$$
\ba
\int_0^1\indc_{2A-1<(h-h')/2<A}\indc_{h'<1-2A}\indc_{-1+2A<h}\indc_{1/(1-A)<S<1/(1-A+(h-h')/4)}dA
\\
=
\int_0^1\indc_{A<\tfrac12+\tfrac14(h-h')}\indc_{A<\tfrac12(1-h')}\indc_{A<\tfrac12(1+h)}\indc_{1-\tfrac1S+\tfrac14(h-h')<A<1-\tfrac1S}dA
\\
=
((\tfrac12+\tfrac14(h-h'))\wedge\tfrac12(1-h')\wedge\tfrac12(1+h)\wedge(1-\tfrac1S)
\\
-(1-\tfrac1S+\tfrac14(h-h'))\vee 0)_+
\\
=((\tfrac12-\tfrac14|h+h'|-\tfrac14(h'-h))\wedge(1-\tfrac1S)-(1-\tfrac1S-\tfrac14(h'-h)|)\vee 0)_+\,.
\ea
$$
Alternatively, using the relations $(a+c)\vee(b+c)=a\vee b+c$, $(a+c)\wedge (b+c)=a\wedge b+c$, $(-a)\wedge(-b)=-(a\vee b)$ and
$a\vee b+a\wedge b=a+b$, we find that
$$
\ba
((\tfrac12-\tfrac14|h+h'|-\tfrac14(h'-h))\wedge(1-\tfrac1S)-(1-\tfrac1S-\tfrac14(h'-h))\vee 0)_+
\\
=
((-\tfrac12-\tfrac14|h+h'|)\wedge(\tfrac14(h'-h)-\tfrac1S)\!+\!1\!-\!\tfrac14(h'-h)-(1-\tfrac14(h'-h))\vee\tfrac1S+\tfrac1S)_+
\\
=
((1-\tfrac14(h'-h))\wedge\tfrac1S-(\tfrac12+\tfrac14|h+h'|)\vee(\tfrac1S-\tfrac14(h'-h)))_+
\ea
$$
so that
$$
\ba
III_1&=\tfrac{12}{\pi^2}\iint f(S,h)\Bigg(\int_0^1\indc_{2A-1<(h-h')/2<A}\indc_{h'<1-2A}\indc_{-1+2A<h}
\\
&\qquad\qquad\times\indc_{1/(1-A)<S<1/(1-A+(h-h')/4)}dA\Bigg)\indc_{-1<h<h'<1}\tfrac{dSdh}{|h-h'|}
\\
&=\tfrac{12}{\pi^2}\iint f(S,h)\indc_{-1<h<h'<1}\Big((S-\tfrac14S(h'-h))\wedge 1
\\
&\qquad\qquad- (1-\tfrac14S(h'-h))\vee(\tfrac{S}2+\tfrac14S|h+h'|)\Big)_+\frac{dSdh}{S|h-h'|}\,.
\ea
$$
With the notation $\zeta=\tfrac12|h+h'|$, we recast this last expression as
$$
III_1\!=\!\tfrac{6}{\pi^2}\iint\! f(S,h)\indc_{-1<h<h'<1}\left((S\!-\!\tfrac12S\eta)\!\wedge\! 1
	\!-\! (1\!-\!\tfrac12S\eta)\!\vee\!(\tfrac12S\!+\!\tfrac12S\zeta)\right)_+\frac{dSdh}{S\eta}\,.
$$

The computation of $III_2$ is fairly similar. Starting with the same substitution as for $III_1$, we obtain
$$
\ba
III_2=\tfrac{12}{\pi^2}\int_0^1\Bigg(\iint f(S,h)\indc_{2A-1<(h-h')/2<A}\indc_{h'<1-2A}\indc_{-1+2A<h}
\\
\indc_{1/(1-A+(h-h')/4)<s<1/(1-A)}\indc_{-1<h'<h<1}\tfrac{dSdh}{|h-h'|}\Bigg)dA
\\
=\tfrac{12}{\pi^2}\iint f(S,h)\Bigg(\int_0^1\indc_{2A-1<(h-h')/2<A}\indc_{h'<1-2A}\indc_{-1+2A<h}
\\
\indc_{1/(1-A+(h-h')/4)<S<1/(1-A)}dA\Bigg)\indc_{-1<h'<h<1}\frac{dSdh}{|h-h'|}\,.
\ea
$$ 
The inner integral is recast as 
$$
\ba
\int_0^1\indc_{2A-1<(h-h')/2<A}\indc_{h'<1-2A}\indc_{-1+2A<h}\indc_{1/(1-A+(h-h')/4)<S<1/(1-A)}dA
\\
=
\int_0^1\indc_{\tfrac12(h-h')<A}\indc_{A<\tfrac12+\tfrac14(h-h')}\indc_{A<\tfrac12(1-h')}\indc_{A<\tfrac12(1+h)}
	\indc_{1-\tfrac1S<A<1-\tfrac1S+\tfrac14(h-h')}dA
\\
=\int_0^1\indc_{\tfrac12(h-h')<A}\indc_{A<\tfrac12-\tfrac14|h+h'|+\tfrac14(h-h')}
	\indc_{1-\tfrac1S<A<1-\tfrac1S+\tfrac14(h-h')}dA
\\
=((\tfrac12-\tfrac14|h+h'|+\tfrac14(h-h'))\wedge(1-\tfrac1S+\tfrac14(h-h'))-(1-\tfrac1S)\vee\tfrac12(h-h'))_+\,.
\ea
$$
As in the case of $III_1$, we have
$$
\ba
(&(\tfrac12-\tfrac14|h+h'|+\tfrac14(h-h'))\wedge(1-\tfrac1S+\tfrac14(h-h'))-(1-\tfrac1S)\vee\tfrac12(h-h'))_+
\\
&\qquad=(1+\tfrac14(h-h')+(-\tfrac12-\tfrac14|h+h'|)\wedge(-\tfrac1S)-(1-\tfrac1S)\vee\tfrac12(h-h'))_+
\\
&\qquad=(1+\tfrac14(h-h')-(\tfrac12+\tfrac14|h+h'|)\vee\tfrac1S
\\
&\qquad\qquad\qquad-(1-\tfrac1S)-\tfrac12(h-h')+(1-\tfrac1S)\wedge\tfrac12(h-h'))_+
\\
&\qquad=(\tfrac1S-\tfrac14(h-h')-(\tfrac12+\tfrac14|h+h'|)\vee(\tfrac1S)+(1-\tfrac1S)\wedge\tfrac12(h-h'))_+
\\
&\qquad=((1-\tfrac14(h-h'))\wedge(\tfrac1S+\tfrac14(h-h'))	-(\tfrac12+\tfrac14|h+h'|)\vee\tfrac1S)_+
\ea
$$
so that
$$
\ba
III_2&=\tfrac{12}{\pi^2}\int_0^1\Bigg(\iint f(S,h)\indc_{2A-1<(h-h')/2<A}\indc_{h'<1-2A}\indc_{-1+2A<h}
\\
&\qquad\qquad\indc_{1/(1-A+(h-h')/4)<S<1/(1-A)}\indc_{-1<h'<h<1}\tfrac{dSdh}{|h-h'|}\Bigg)dA
\\
&=\tfrac{12}{\pi^2}\iint f(S,h)\indc_{-1<h'<h<1}\Big((S-\tfrac14S(h-h'))\wedge(1+\tfrac14S(h-h'))
\\
&\qquad\qquad\qquad\qquad\qquad\qquad-(\tfrac12S+\tfrac14S|h+h'|)\vee 1\Big)_+\frac{dSdh}{S|h-h'|}
\ea
$$ 
In other words
$$
III_2\!=\!\tfrac{6}{\pi^2}\iint\! f(S,h)\indc_{-1<h<h'<1}\left((S\!-\!\tfrac12S\eta)\!\wedge\! (1\!+\!\tfrac12S\eta)
	\!-\!(\tfrac12S\!+\!\tfrac12S\zeta)\!\vee\!1\right)_+\frac{dSdh}{S\eta}\,.
$$

Summing up the contributions $I$, $II$, $III_1$ and $III_2$, we find that
$$
\ba
P_+(S,h|h')&=\frac6{\pi^2S\eta}\indc_{-1<h<h'<1}\indc_{0<S<1}(S\eta)\wedge(1-S)
\\
&+
\frac{6}{\pi^2S\eta}\indc_{-1<h'<h<1}\indc_{S>0}(S\eta-|1-S|)_+
\\
&+
\frac{6}{\pi^2S\eta}\indc_{-1<h<h'<1}\indc_{S>0}
	\left((S\!-\!\tfrac12S\eta)\!\wedge\! 1\!-\! (1\!-\!\tfrac12S\eta)\!\vee\!(\tfrac12S\!+\!\tfrac12S\zeta)\right)_+
\\
&+
\frac{6}{\pi^2S\eta}\indc_{-1<h'<h<1}\indc_{S>0}
	\left((S\!-\!\tfrac12S\eta)\!\wedge\! (1\!+\!\tfrac12S\eta)\!-\!(\tfrac12S\!+\!\tfrac12S\zeta)\!\vee\!1\right)_+\,.
\ea
$$
Observe that
$$
\indc_{S>0}(S\eta-|1-S|)_+=\indc_{\frac1{1+\eta}<S<\frac1{1-\eta}}(S\eta-|1-S|)\,.
$$
On the other hand, if
$$
(S\!-\!\tfrac12S\eta)\!\wedge\! 1\!-\! (1\!-\!\tfrac12S\eta)\!\vee\!(\tfrac12S\!+\!\tfrac12S\zeta)\ge 0
$$
then $(S\!-\!\tfrac12S\eta)-(1\!-\!\tfrac12S\eta)\ge 0$ so that $S\ge 1$; likewise $1\!-\!(\tfrac12S\!+\!\tfrac12S\zeta)\ge 0$ so that $S\le 2$.
Hence
$$
\ba
\indc_{S>0}\left((S\!-\!\tfrac12S\eta)\!\wedge\! 1\!-\! (1\!-\!\tfrac12S\eta)\!\vee\!(\tfrac12S\!+\!\tfrac12S\zeta)\right)_+
\\
=\indc_{1<S<2}\left((S\!-\!\tfrac12S\eta)\!\wedge\! 1\!-\! (1\!-\!\tfrac12S\eta)\!\vee\!(\tfrac12S\!+\!\tfrac12S\zeta)\right)_+\,.
\ea
$$
By the same token
$$
(S\!-\!\tfrac12S\eta)\!\wedge\! (1\!+\!\tfrac12S\eta)\!-\!(\tfrac12S\!+\!\tfrac12S\zeta)\!\vee\!1\ge 0
$$
implies that $(S\!-\!\tfrac12S\eta)-1\ge 0$, so that $S\ge 1$, and hence 
$$
\ba
\indc_{S>0}\left((S\!-\!\tfrac12S\eta)\!\wedge\! (1\!+\!\tfrac12S\eta)\!-\!(\tfrac12S\!+\!\tfrac12S\zeta)\!\vee\!1\right)_+
\\
=\indc_{S>1}\left((S\!-\!\tfrac12S\eta)\!\wedge\! (1\!+\!\tfrac12S\eta)\!-\!(\tfrac12S\!+\!\tfrac12S\zeta)\!\vee\!1\right)_+\,. 
\ea
$$
Finally
$$
\ba
P_+(S,h|h')&=\frac6{\pi^2S\eta}\indc_{-1<h<h'<1}\indc_{0<S<1}(S\eta)\wedge(1-S)
\\
&+\frac{6}{\pi^2S\eta}\indc_{-1<h'<h<1}\indc_{\frac1{1+\eta}<S<\frac1{1-\eta}}(S\eta-|1-S|)
\\
&+\frac{6}{\pi^2S\eta}\indc_{-1<h<h'<1}\indc_{1<S<2}
	\left((S\!-\!\tfrac12S\eta)\!\wedge\! 1\!-\! (1\!-\!\tfrac12S\eta)\!\vee\!(\tfrac12S\!+\!\tfrac12S\zeta)\right)_+
\\
&+\frac{6}{\pi^2S\eta}\indc_{-1<h'<h<1}\indc_{S>1}
	\left((S\!-\!\tfrac12S\eta)\!\wedge\! (1\!+\!\tfrac12S\eta)\!-\!(\tfrac12S\!+\!\tfrac12S\zeta)\!\vee\!1\right)_+\,.
\ea
$$
By formula \ref{TransitLimit},
$$
\bT_{A,B,Q,-}(-h')=(L,-h)\Longleftrightarrow\bT_{A,B,Q,+}(h')=(L,h)
$$
so that
$$
P(S,h|h')=\tfrac12\left(P_+(S,h|h')+P_+(S,-h|-h')\right)\,,
$$
thereby leading to formula (\ref{LimTransitProba}).

\subsection{Proof of the simplified formula (\ref{FlaTransitProbaSimpl})}
\label{SS-FlaPSimpl}

Assume that $h>|h'|$ so that 
$$
\left\{
\begin{array}{l}
\eta=\tfrac12(h-h')
\\ \\
\zeta=\tfrac12(h+h')
\end{array}
\right.
\quad\Rightarrow 
\left\{
\begin{array}{l}
h\,=(\zeta+\eta)
\\ \\
h'=(\zeta-\eta)
\end{array}
\right.
$$
and denote
$$
L=S-\tfrac12S\eta\,,\quad M=\tfrac12S+\tfrac12S\zeta\,,\quad N=\tfrac12S\eta\,.
$$
Whenever $S\ge 1$, the reader will easily check that (\ref{FlaTransitProba}) can be written as
$$
\tfrac{2\pi^2}{3}NP(S,h|h')=(1+N-L)_++(L\wedge(1+N)-M\vee 1)_++(L\wedge 1-M\vee (1-N))_+\,.
$$

Since $L\ge M$, the expression above can be reduced after checking the three cases $L\le 1$, $1<L<1+N$ and $L\ge 1+N$. 
One finds that
$$
(1+N-L)_++(L\wedge(1+N)-M\vee 1)_+=(1+N-M\vee(L\wedge 1))_+\,.
$$

Then
$$
\ba
\tfrac{2\pi^2}{3}NP(S,h|h')&=(1+N-L)_++(L\wedge(1+N)-M\vee 1)_+
\\
&+(L\wedge 1-M\vee (1-N))_+
\\
&=(1+N-M\vee(L\wedge 1))_++(L\wedge 1-M\vee (1-N))_+
\\
&=\left\{
\begin{array}{ll}
0\quad&\hbox{ if }M\ge 1+N\,,
\\
1+N-M&\hbox{ if }1-N<M<1+N\,,
\\
2N&\hbox{ if }M\le 1-N\,.
\end{array}
\right.
\ea
$$
Since
$$
\left\{\begin{array}{l}
M+N=\tfrac12S(1+h)
\\	\\
M-N=\tfrac12S(1+h')
\end{array}
\right.
$$
the formulas above can be recast as
$$
\tfrac{2\pi^2}{3}NP(S,h|h')
=\left\{
\begin{array}{ll}
0\quad&\hbox{ if }\tfrac12S(1+h')\ge 1\,,
\\ \\
1-\tfrac12S(1+h')&\hbox{ if }\tfrac12S(1+h')<1<\tfrac12S(1+h)\,,
\\ \\
\tfrac12S(h-h')&\hbox{ if }1\le\tfrac12S(1+h)\,,
\end{array}
\right.
$$
which holds whenever $S\ge 1$.

On the other hand, if $S<1$, the last two terms on the right hand side of (\ref{FlaTransitProba}) vanish identically so that
$$
\tfrac{2\pi^2}{3}NP(S,h|h')=(2N)\wedge(1-S)+(2N-(1-S))_+=2N=\tfrac12S(h-h').
$$
Putting together these last two formulas leads to (\ref{FlaTransitProbaSimpl}).

\newpage
\subsection{Proof of Corollary \ref{C-LimTr}}
\label{SS-CorLimTr}

As for the statements in Corollary \ref{C-LimTr}, observe that the symmetries (\ref{SymP}) and the fact that $P$ is piecewise 
continuous on $\bR_+\times[-1,1]\times[1,1]$ follow from (\ref{FlaTransitProba}). The first identity in (\ref{Int-Pdsdh=1}) follows from 
the fact that, by definition, $(s,h)\mapsto P(s,h|h')$ is a probability density, while the second identity there follows from the first and 
the symmetry $P(s,h|h')=P(s,h'|h)$. 

If $S\ge 4$ and $h,h'\in]-1,1[$ satisfy $|h'|\le h$, the simplified formula (\ref{FlaTransitProbaSimpl}) implies that
$$
P(S,h|h')\le\frac{6}{\pi^2S(h-h')}\indc_{1+h'<\frac2S}\,.
$$
On the other hand, $1+h'<\frac2S$ and $S\ge 4$ imply that $h'<-\tfrac12$, so that $h\ge|h'|>\tfrac12$; therefore $h-h'>1$ and the
inequality above entails (\ref{BoundP}). 

Starting from (\ref{BoundP}), observe that, because of the symmetries (\ref{SymP}),
$$
\ba
\int_{-1}^1\int_{-1}^1P(s,h|h')dhdh'&=4\iint_{0\le|h'|<h\le 1}P(s,h|h')dhdh'
\\
&\le\frac{24}{\pi^2S}\iint_{0\le|h'|<h\le 1}\indc_{1+h'<\frac2S}dhdh'=\frac{48}{\pi^2S^3}\,,
\ea
$$
which is (\ref{BoundIntP}).


\section{Proof of Theorem \ref{T-PropEquil}}
\label{S-PropEquil}

The existence of the integral defining $E$ follows from the positivity of $P(s,h|h')$ and the second identity in (\ref{Int-Pdsdh=1}).

That $E(0,h)=1$ for $|h|\le 1$ follows from the formula defining $E$ and the second identity in (\ref{Int-Pdsdh=1}): this establishes the 
first part of statement 1). 

The definition of $E$ and the first formula in statement 1) show that
$$
-\d_sE(s,h)=\int_{-1}^12P(2s,h|h')dh'=\int_{-1}^12P(2s,h|h')E(0,h')dh'\,,\quad s>0\,,\,\,|h|\le 1\,.
$$
Using again the positivity of $P(s,h|h')$ and the second identity in (\ref{Int-Pdsdh=1}) shows that, for each $h\in[-1,1]$, the function 
$(s,h')\mapsto P(s,h|h')$ belongs to $L^1(\bR_+\times[-1,1])$. Therefore $E(s,h)\to 0$ for each $h\in[-1,1]$ as $s\to+\infty$. Thus each 
function $F\equiv F(s,h)$ of the form $F(s,h)=CE(s,h)$ satisfies both conditions in statement 2).

Conversely, let $F\equiv F(s,h)$ satisfy the conditions in statement 2), and let $\Phi(h):=F(0,h)$ for a.e. $h\in[-1,1]$. Integrating both
sides of the differential equation satisfied by $F$ in $s\in\bR_+$ yields
\begin{equation}
\lb{EqPhi}
\Phi(h)=\int_{-1}^1\Pi(h|h')\Phi(h')dh'
\end{equation}
since $F(s,h)$ vanishes as $s\to+\infty$, with 
$$
\Pi(h|h')=\int_0^\infty P(S,h|h')dS\,.
$$

Multiplying each side of the identity (\ref{EqPhi}) by $\Phi(h)$, and integrating in $h\in[-1,1]$, we see that
\be\lb{EigenEq}
\int_{-1}^1\Phi(h)^2dh=\int_{-1}^1\int_{-1}^1\Pi(h|h')\Phi(h)\Phi(h')dhdh'\,.
\ee

Observe that, by (\ref{SymP})
$$
\Pi(h'|h)=\Pi(h|h')\,,\quad\hbox{ for a.e. }h,h'\in[-1,1]\,,
$$
and
$$
\int_{-1}^1\Pi(h|h')dh'=\int_{-1}^1\Pi(h'|h)dh'=1\,,\quad\hbox{ for a.e. }h\in[-1,1]\,,
$$
in view of (\ref{Int-Pdsdh=1}). 

Therefore
$$
\int_{-1}^1\Phi(h)^2dh=\int_{-1}^1\left(\int_{-1}^1\Pi(h|h')dh'\right)\Phi(h)^2dh
	=\int_{-1}^1\left(\int_{-1}^1\Pi(h|h')dh\right)\Phi(h')^2dh\,,
$$
so that (\ref{EigenEq}) becomes
$$
\ba
0&=
\int_{-1}^1\Phi(h)^2dh-\int_{-1}^1\int_{-1}^1\Pi(h|h')\Phi(h)\Phi(h')dhdh'
\\
&=
\int_{-1}^1\int_{-1}^1\Pi(h|h')\tfrac12(\Phi(h)^2+\Phi(h')^2)dhdh'-\int_{-1}^1\int_{-1}^1\Pi(h|h')\Phi(h)\Phi(h')dhdh'
\\
&=
\int_{-1}^1\int_{-1}^1\Pi(h|h')\tfrac12(\Phi(h)-\Phi(h'))^2dh\,.
\ea
$$
Since $\Pi(h|h')>0$ a.e. in $h,h'\in[-1,1]$ as can be seen from the explicit formula (\ref{FlaTransitProba-hh'}), this last equality implies 
that
$$
\Phi(h)=C\quad\hbox{ a.e. in }h\in[-1,1]\,,
$$
for some nonnegative constant $C$. Therefore
$$
-\d_sF(s,h)=C\int_{-1}^12P(2s,h|h')dh'\,,\,\,\,\hbox{ and }\lim_{s\to+\infty}F(s,h)=0\quad\hbox{Êa.e. in }h\in[-1,1]\,,
$$
so that
$$
F(s,h)=C\int_{2s}^\infty\int_{-1}^1P(\tau,h|h')dh'd\tau=CE(s,h)\,,
$$
which proves the uniqueness part of statement 2).

Now for statement 3); by definition of $P(s,h|h')$, for each $t>0$
$$
\ba
\int_{-1}^1E(s,h)dh&=\int_{2s}^\infty\int_{-1}^1\int_{-1}^1P(t,h|h')dtdhdh'
\\
&=\lim_{\eps\to 0^+}\frac1{|\ln\eps|}\int_\eps^{1/4}\int_{-1}^1\indc_{[2s,+\infty)\times[-1,1]}(T_r(h',\om))dh'\frac{dr}r
\\
&=\lim_{\eps\to 0^+}\frac1{|\ln\eps|}\int_\eps^{1/4}\int_{-1}^1\indc_{[2s,+\infty)}(2r\tau_r(Y_r(h',\om))dh'\frac{dr}r
\\
&=2\lim_{\eps\to 0^+}\frac1{|\ln\eps|}\int_\eps^{1/4}\nu_r(\{(x,\om)\in\Gamma_r^+/\bZ^2\,|\,2r\tau_r(x,\om)\ge 2s\}\frac{dr}r
\ea
$$
where $\nu_r$ is the probability measure on $\Gamma_r^+/\bZ^2$ that is proportional to $\om\cdot n_xdxd\om$. Using formula
(1.3) in \cite{BocaZaharescu}, which is a straightforward consequence of variant of Santal\'o's formula established in Lemma 3 of
\cite{Dumas2Golse1996}, we conclude that
$$
\int_{-1}^1E(s,h)dh=-2p'(2s)\,,\quad s\ge 0\,,
$$
which is the first formula in statement 3). The second formula there is a consequence of the expression of $p''(s)$ as a power
series in $1/s$ given in formula (1.5) of \cite{BocaZaharescu}.

Finally, we establish the second formula in statement 1). Indeed
$$
\int_0^\infty\int_{-1}^1E(s,h)dhds=\int_0^\infty-2p'(2s)ds=p(0)=1\,,
$$
while the first equality in that formula follows from the identity defining $E$ and Fubini's theorem.


\section{Proof of Theorem \ref{T-HThm}}
\label{S-HThm}

If $F$ is a solution of (\ref{EqLim}), we set $f=F/E$, so that, observing that $E(0,h')=1$ for $h'\in[-1,1]$,
$$
\ba
E(s,h)(\d_t+\om\cdot\grad_x-\d_s)&f(t,x,\om,s,h)-f(t,x,\om,s,h)\d_sE(s,h)&
\\
&=
\int_{-1}^12P(2s,h|h')f(t,x,R[\theta(h')],0,h')dh'\,,
\ea
$$
Since
$$
\d_sE(s,h)=-\int_{-1}^12P(2s,h|h')dh'\,,
$$
the equation above can be put in the form
$$
\ba
E(s,h)&(\d_t+\om\cdot\grad_x-\d_s)f(t,x,\om,s,h)
\\
&=
\int_{-1}^12P(2s,h|h')\left(f(t,x,R[\theta(h')],0,h')-f(t,x,\om,s,h)\right)dh'\,.
\ea
$$

Next we multiply both sides of the equation above by $\bh'(f)-\bh'(1)$: the left hand side becomes
$$
\ba
(\bh'(f(t,x,\om,s,h))-\bh'(1))E(s,h)(\d_t+\om\cdot\grad_x-\d_s)f(t,x,\om,s,h)
\\
=
(\d_t\!+\!\om\cdot\grad_x\!-\!\d_s)\left(E(s,h)(\bh(f(t,x,\om,s,h))-\bh(1)-\bh'(1)(f(t,x,\om,s,h)-1))\right)
\\
+
\d_sE(s,h)(\bh(f(t,x,\om,s,h))-\bh(1)-\bh'(1)(f(t,x,\om,s,h)-1))
\\
=
(\d_t\!+\!\om\cdot\grad_x\!-\!\d_s)\left(E(s,h)(\bh(f(t,x,\om,s,h))-\bh(1)-\bh'(1)(f(t,x,\om,s,h)-1))\right)
\\
-\int_{-1}^12P(2s,h|h')(\bh(f(t,x,\om,s,h))-\bh(1)-\bh'(1)(f(t,x,\om,s,h)-1))dh'\,.
\ea
$$
Integrating in $(\om,s,h)$ transforms this expression into
$$
\ba
{}&\d_t\int_{\bS^1}\int_0^\infty\int_{-1}^1(\bh(f)-\bh(1)-\bh'(1)(f-1))(t,x,\om,s,h)E(s,h)dhdsd\om
\\
&+
\Div_x\int_{\bS^1}\int_0^\infty\int_{-1}^1\om(\bh(f)-\bh(1)-\bh'(1)(f-1))(t,x,\om,s,h)E(s,h)dhdsd\om
\\
&+\int_{\bS^1}\int_{-1}^1\left(\bh(f)-\bh(1)-\bh'(1)(f-1))\right)(t,x,\om,0,h)dhd\om 
\\
&-\int_{\bS^1}\int_0^\infty\int_{-1}^1\int_{-1}^12P(2s,h|h')(\bh(f)\!-\!\bh(1)\!-\!\bh'(1)(f\!-\!1))(t,x,\om,s,h)dh'dhdsd\om\,.
\ea
$$
Using the relation (\ref{Int-Pdsdh=1}) simplifies this term into
$$
\ba
{}&\d_t\int_{\bS^1}\int_0^\infty\int_{-1}^1(\bh(f)-\bh(1)-\bh'(1)(f-1))(t,x,\om,s,h)E(s,h)dhdsd\om
\\
&+
\Div_x\int_{\bS^1}\int_0^\infty\int_{-1}^1\om(\bh(f)-\bh(1)-\bh'(1)(f-1))(t,x,\om,s,h)E(s,h)dhdsd\om
\\
&+\int_{\bS^1}\int_{-1}^1(\bh(f)-\bh'(1)f)(t,x,\om,0,h')dh'd\om 
\\
&-\int_{\bS^1}\int_0^\infty\int_{-1}^1\int_{-1}^12P(2s,h|h')(\bh(f)-\bh'(1)f)(t,x,\om,s,h)dh'dhdsd\om\,.
\ea
$$

On the other hand, multiplying the right hand side by $\bh'(f(t,x,\om,s,h))-\bh'(1)$ and integrating in $(\om,s,h)$ leads to
$$
\ba
{}&\int_{\bS^1}\int_0^\infty\int_{-1}^1\int_{-1}^12P(2s,h|h')(f(t,x,R[\theta(h')]\om,0,h')-f(t,x,\om,s,h))
\\
&\qquad\qquad\qquad\qquad\qquad\qquad\qquad\qquad\qquad\times(\bh'(f(t,x,\om,s,h))-\bh'(1))dh'dhdsd\om
\\
&=\int_{\bS^1}\int_0^\infty\int_{-1}^1\int_{-1}^12P(2s,h|h')(f(t,x,R[\theta(h')]\om,0,h')-f(t,x,\om,s,h))
\\
&\qquad\qquad\qquad\qquad\qquad\qquad\qquad\qquad\qquad\qquad\qquad\times\bh'(f(t,x,\om,s,h))dh'dhdsd\om
\\
&+\bh'(1)\int_{\bS^1}\int_0^\infty\int_{-1}^1\int_{-1}^12P(2s,h|h')f(t,x,\om,s,h)dh'dhdsd\om
\\
&\qquad\qquad\qquad\qquad\qquad\qquad\qquad\qquad\qquad-\bh'(1)\int_{\bS^1}\int_{-1}^1f(t,x,\om,0,h')dh'd\om\,.
\ea
$$
after substituting $\om$ for $R[\theta(h')]\om$ in the last integral above.

Putting together the left- and right-hand sides, we arrive at the equality
$$
\ba
{}&\d_t\int_{\bS^1}\int_0^\infty\int_{-1}^1(\bh(f)-\bh(1)-\bh'(1)(f-1))(t,x,\om,s,h)E(s,h)dhdsd\om
\\
&+
\Div_x\int_{\bS^1}\int_0^\infty\int_{-1}^1\om(\bh(f)-\bh(1)-\bh'(1)(f-1))(t,x,\om,s,h)E(s,h)dhdsd\om
\\
&+\int_{\bS^1}\int_{-1}^1(\bh(f)-\bh'(1)f)(t,x,\om,0,h')dh'd\om 
\\
&-\int_{\bS^1}\int_0^\infty\int_{-1}^1\int_{-1}^12P(2s,h|h')(\bh(f)-\bh'(1)f)(t,x,\om,s,h)dh'dhdsd\om 
\\
&=\int_{\bS^1}\int_0^\infty\int_{-1}^1\int_{-1}^12P(2s,h|h')(f(t,x,R[\theta(h')]\om,0,h')-f(t,x,\om,s,h))
\\
&\qquad\qquad\qquad\qquad\qquad\qquad\qquad\qquad\qquad\qquad\times\bh'(f(t,x,\om,s,h))dh'dhdsd\om
\\
&+\bh'(1)\int_{\bS^1}\int_0^\infty\int_{-1}^1\int_{-1}^12P(2s,h|h')f(t,x,\om,s,h)dh'dhdsd\om
\\
&\qquad\qquad\qquad\qquad\qquad\qquad\qquad\qquad\qquad-\bh'(1)\int_{\bS^1}\int_{-1}^1f(t,x,\om,0,h')dh'd\om\,.
\ea
$$
All the terms with a factor $\bh'(1)$ in the integral part of the equality above compensate, so that the equality above reduces to
$$
\ba
{}&\d_t\int_{\bS^1}\int_0^\infty\int_{-1}^1(\bh(f)-\bh(1)-\bh'(1)(f-1))(t,x,\om,s,h)E(s,h)dhdsd\om
\\
&+
\Div_x\int_{\bS^1}\int_0^\infty\int_{-1}^1\om(\bh(f)-\bh(1)-\bh'(1)(f-1))(t,x,\om,s,h)E(s,h)dhdsd\om
\\
&+\int_{\bS^1}\int_{-1}^1\bh(f)(t,x,\om,0,h')dh'd\om 
\\
&-\int_{\bS^1}\int_0^\infty\int_{-1}^1\int_{-1}^12P(2s,h|h')\bh(f)(t,x,\om,s,h)dh'dhdsd\om 
\\
&-\int_{\bS^1}\int_0^\infty\int_{-1}^1\int_{-1}^12P(2s,h|h')(f(t,x,R[\theta(h')]\om,0,h')-f(t,x,\om,s,h)
\\
&\qquad\qquad\qquad\qquad\qquad\qquad\qquad\qquad\qquad\qquad\times\bh'(f(t,x,\om,s,h))dh'dhdsd\om=0\,.
\ea
$$

Using again the relation (\ref{Int-Pdsdh=1}) and the substitution $\om\mapsto R[\th(h')]\om$, one has
$$
\ba
\int_{\bS^1}\int_{-1}^1&\bh(f)(t,x,\om,0,h')dh'd\om 
\\
&=
\int_{\bS^1}\int_0^\infty\int_{-1}^1\int_{-1}^12P(2s,h|h')\bh(f)(t,x,\om,0,h')dh'dhdsd\om 
\\
&=\int_{\bS^1}\int_0^\infty\int_{-1}^1\int_{-1}^12P(2s,h|h')\bh(f)(t,x,R[\th(h')]\om,0,h')dh'dhdsd\om 
\ea
$$
so that the previous identity can be put in the form
$$
\ba
{}&\d_t\int_{\bS^1}\int_0^\infty\int_{-1}^1(\bh(f)-\bh(1)-\bh'(1)(f-1))(t,x,\om,s,h)E(s,h)dhdsd\om
\\
&+
\Div_x\int_{\bS^1}\int_0^\infty\int_{-1}^1\om(\bh(f)-\bh(1)-\bh'(1)(f-1))(t,x,\om,s,h)E(s,h)dhdsd\om
\\
&+\int_{\bS^1}\int_0^\infty\int_{-1}^1\int_{-1}^12P(2s,h|h')\Big(\bh(f)(t,x,R[\th(h')]\om,0,h')-\bh(f)(t,x,\om,s,h)
\\
&\qquad-(f(t,x,R[\theta(h')]\om,0,h')-f(t,x,\om,s,h))\bh'(f(t,x,\om,s,h)\Big)dh'dhdsd\om=0\,. 
\ea
$$
Integrating both sides of this identity in $x\in\bT^2$, we finally obtain
$$
\frac{d}{dt}H_\bh(fE|E)+\int_{\bT^2}D_\bh(f)(t,x)dx=0\,,
$$
with $H_\bh(F|E)$ and $D_\bh(f)$ defined as in the statement of Theorem \ref{T-HThm}.


\section{Proof of Theorem \ref{T-InstabLocEquil}}
\label{S-InstabLocEquil}

Inserting $F(t,x,\om,s,h)=f(t,x,\om)E(s,h)$ in (\ref{EqLim}), one finds that 
$$
\ba
E(s,h)(\d_t+\om\cdot\grad_x)f(t,x,\om)&-f(t,x,\om)\d_sE(s,h)
\\
&=\int_{-1}^12P(2s,h|h')f(t,x,R[\th(h')]\om)dh'
\ea
$$
so that
$$
\ba
E(s,h)(\d_t+\om\cdot\grad_x)&f(t,x,\om)
\\
&=\int_{-1}^12P(2s,h|h')\left(f(t,x,R[\th(h')]\om)-f(t,x,\om)\right)dh'
\ea
$$
since 
$$
-\d_sE(s,h)=\int_{-1}^12P(2s,h|h')dh'\,,\quad s>0\,,\,\,|h|\le 1
$$
--- see Theorem \ref{T-PropEquil}. 

Integrating in $h\in [-1,1]$ both sides of the penultimate equality, we obtain
$$
\ba
\int_{-1}^1E(s,h)dh&(\d_t+\om\cdot\grad_x)f(t,x,\om)
\\
&=\int_{-1}^1\int_{-1}^12P(2s,h|h')\left(f(t,x,R[\th(h')]\om)-f(t,x,\om)\right)dh'\,.
\ea
$$
Since
$$
\int_{-1}^1E(s,h)dh\sim\frac{1}{\pi^2s^2}\quad\hbox{Êas }s\to+\infty
$$
by Theorem \ref{T-PropEquil}, while
$$
\int_{-1}^1\int_{-1}^1P(S,h|h')dhdh'\le\frac{C'}{(1+S)^3}
$$
by (\ref{BoundIntP}), one has
$$
\ba
\int_{-1}^1E(s,h)dh(\d_t+\om\cdot\grad_x)f(t,x,\om)\sim\frac{1}{\pi^2s^2}(\d_t+\om\cdot\grad_x)f(t,x,\om)\,,\,\,\hbox{ and }
\\
\left|\int_{-1}^1\int_{-1}^12P(2s,h|h')\left(f(t,x,R[\th(h')]\om)-f(t,x,\om)\right)dh'dh\right|
\\
\le\frac{4C'}{(1+2s)^3}\|f(t,\cdot,\cdot)\|_{L^\infty(\bT^2\times\bS^1)}\,.
\ea
$$
Hence
$$
\left\{
\ba
{}&(\d_t+\om\cdot\grad_x)f(t,x,\om)=0\,,\,\,\hbox{ and }
\\
&\int_{-1}^1\d_sE(s,h')\left(f(t,x,R[\th(h')]\om)-f(t,x,\om)\right)dh'=0
\ea
\right.
$$
Integrating the second equality in $s>0$, and observing that $E\rstr_{s=0}=1$ while $E(s,h)\to 0$ uniformly in $h\in[-1,1]$ as 
$s\to+\infty$, we see that
$$
\ba
\int_0^\infty&\left(\int_{-1}^1\d_sE(s,h')\left(f(t,x,R[\th(h')]\om)-f(t,x,\om)\right)dh'\right)ds
\\
&=
\int_{-1}^1\left(\int_0^\infty\d_sE(s,h')ds\right)\left(f(t,x,R[\th(h')]\om)-f(t,x,\om)\right)ds
\\
&=
\int_{-1}^1\left(f(t,x,R[\th(h')]\om)-f(t,x,\om)\right)dh'
=0
\ea
$$
or, in other words
$$
f(t,x,\om)=\tfrac12\int_{-1}^1f(t,x,R[\th(h')]\om)dh'=:\phi(t,x)\,.
$$
On the other hand, the first equation implies that
$$
(\d_t+\om\cdot\grad_x)\phi(t,x)=0
$$
so that
$$
\d_t\phi=\d_{x_j}\phi=0\,,\quad j=1,2\,.
$$
Hence $\phi$ is a constant, which implies in turn that $f(t,x,\om)=\frac{F(t,x,\om,s,h)}{E(s,h)}$ is a constant.


\section{Proof of Theorem \ref{T-LongTime}}
\label{S-LongTime}

Let $f^{in}\in L^\infty(\bT^2\times\bS^1)$ be such that $f^{in}(x,\om)\ge 0$ a.e. in $(x,\om)\in\bT^2\times\bS^1$, and let $F$ be the
solution of the Cauchy problem (\ref{EqLim}). Define $F_0(t,x,\om,s,h)=f^{in}(x,\om)E(s,h)$ and set $K_jF_0=F_j$ for $j\in\bN$,
where we recall that $(K_t)_{t\ge 0}$ is the evolution semigroup associated to the Cauchy problem (\ref{EqLim}).

\smallskip
\noindent
\underline{Step 1:}

Assume first that $\grad^m_xf^{in}\in L^\infty(\bT^2\times\bS^1)$ for all $m\ge 0$, so that $F$ solves (\ref{EqLim}) in the classical
sense. Then, with $\bh(z)=\tfrac12z^2$, one has
$$
H(F_{n+1}|E)+\sum_{j=0}^n\left(H(F_j|E)-H(F_{j+1}|E)\right)=H(F_0|E)
$$
for each $n\ge 0$, and since
$$
H(F_j|E)-H(F_{j+1}|E)=H(F_j|E)-H(K_1F_j|E)\ge 0
$$
by Theorem \ref{T-HThm}, one has
$$
H(F_j|E)-H(F_{j+1}|E)\to 0\hbox{ as }j\to+\infty\,.
$$

Now, 
$$
\ba
{}&H(F_j|E)-H(F_{j+1}|E)=\int_0^1\int_{\bT^2}D(K_tF_j/E)dxdt
\\
&=\int_0^1\int_{\bT^2}\int_{\bS^1}\int_0^\infty\iint_{[-1,1]^2}P(2s,h|h')\Phi_j(t,x,\om,s,h,h')^2dhdh'dsd\om dxdt
\ea
$$
with the notation
$$
\Phi_j(t,x,\om,s,h,h')=\frac{K_tF_j(x,\om,s,h)}{E(s,h)}-K_tF_j(x,R[\th(h')]\om,0,h')
$$
for a.e. $(t,x,\om,s,h,h')\in[0,1]\times\bT^2\times\bS^1\times\bR_+\times[-1,1]^2$.

In view of properties 1) and 3) of the evolution semigroup $(K_t)_{t\ge 0}$ recalled in section \ref{S-PropLimit}, one has
$$
0\le\frac{K_tF_j(x,\om,s,h)}{E(s,h)}\le\|f^{in}\|_{L^\infty(\bT^2\times\bS^1)}
$$
for a.e. $(x,\om,s,h)\in\bT^2\times\bS^1\times\bR_+\times[-1,1]$ and each $t\ge 0$, so that, up to extraction of a subsequence,
$$
F_{j_k}/E\wto F/E\hbox{ in }L^\infty(\bT^2\times\bS^1\times\bR_+\times[-1,1])\hbox{ weak-*}
$$
as $j_k\to+\infty$, by the Banach-Alaoglu theorem. Besides, the estimate (\ref{BoundP}) implies that
$$
0\le(\d_t+\om\cdot\grad_x-\d_s)K_tF_{j_k}\le\frac{4C''}{1+2s}\|f^{in}\|_{L^\infty(\bT^2\times\bS^1)}
$$
so that, by the usual trace theorem for the advection operator $\d_t+\om\cdot\grad_x-\d_s$ (see for instance \cite{Cessenat1984})
$$
K_tF_{j_k}(x,R[\th(h')]\om,0,h')\wto K_tF(x,R[\th(h')]\om,0,h')
$$
in $L^\infty([0,1]\times\bT^2\times\bS^1\times[-1,1])$  weak-*. In particular
$$
\Phi_{j_k}\wto\Phi\hbox{ in }L^\infty([0,1]\times\bT^2\times\bS^1\times\bR_+\times[-1,1]^2)\hbox{ weak-*}
$$
with
$$
\Phi(t,x,\om,s,h,h')=\frac{K_tF(x,\om,s,h)}{E(s,h)}-K_tF(x,R[\th(h')]\om,0,h')
$$
a.e. in $[0,1]\times\bT^2\times\bS^1\times\bR_+\times[-1,1]^2$. By convexity and weak limit
$$
\ba
\int_0^1\int_{\bT^2}\int_{\bS^1}\int_0^\infty\int_{-1}^1\int_{-1}^1P(2s,h|h')\Phi(t,x,\om,s,h,h')^2dh'dhdsd\om dxdt
\\
=
\int_0^1\int_{\bT^2}D(K_tF/E)dxdt\le\lim_{k\to\infty}\int_0^1\int_{\bT^2}D(K_tF_{j_k}/E)dxdt=0\,,
\ea
$$
so that
$$
\Phi(t,x,\om,s,h,h')=\frac{K_tF(x,\om,s,h)}{E(s,h)}-K_tF(x,R[\th(h')]\om,0,h')=0
$$
a.e. on $[0,1]\times\bT^2\times\bS^1\times\bR_+\times[-1,1]^2$. Averaging both sides of this identity in $h'\in[-1,1]$ shows that
$$
\frac{K_tF(x,\om,s,h)}{E(s,h)}=\tfrac12\int_{-1}^1K_tF(x,R[\th(h')]\om,0,h')dh'=:f(t,x,\om)
$$
for a.e. $[0,\tau]\times\bT^2\times\bS^1\times\bR_+\times[-1,1]$, i.e.
$$
K_tF(x,\om,s,h)=f(t,x,\om)E(s,h)\,.
$$
By Theorem \ref{T-InstabLocEquil}, one has $f(t,x,\om)=C$ a.e. in $(t,x,\om)\in[0,1]\times\bT^2\times\bS^1$ for some constant 
$C\ge 0$, so that 
$$
F(x,\om,s,h)=CE(s,h)\hbox{ a.e. in }(x,\om,s,h)\in\bT^2\times\bS^1\times\bR_+\times[-1,1]\,.
$$

Let us identify the constant $C$. Property 4) of the semigroup $(K_t)_{t\ge 0}$ recalled in section \ref{S-PropLimit} implies that
$$
\ba
\iint_{\bT^2\times\bS^1}&\int_0^\infty\int_{-1}^1F_j(x,\om,s,h)dhdsdxd\om 
\\
&=\iint_{\bT^2\times\bS^1}f^{in}(x,\om)dxd\om\int_0^\infty\!\!\int_{-1}^1E(s,h)dhds=\!\iint_{\bT^2\times\bS^1}f^{in}(x,\om)dxd\om\,.
\ea
$$
Since $F_{j_k}/E\wto F/E$ in $L^\infty(\bT^2\times\bS^1\times\bR_+\times[-1,1])$ weak-* as $j_k\to+\infty$, 
$$
\ba
\iint_{\bT^2\times\bS^1}f^{in}(x,\om)dxd\om=\iint_{\bT^2\times\bS^1}\int_0^\infty\int_{-1}^1F(x,\om,s,h)dhdsdxd\om&
\\
=\iint_{\bT^2\times\bS^1}\int_0^\infty\int_{-1}^1CE(s,h)dhdsdxd\om=2\pi C&\,.
\ea
$$
Hence 
$$
F_{j_k}/E\wto C=\tfrac1{2\pi}\iint_{\bT^2\times\bS^1}f^{in}(x,\om)dxd\om
$$
in $L^\infty(\bT^2\times\bS^1\times\bR_+\times[-1,1])$ weak-*.

Since the sequence $(F_j/E)_{j\ge 0}$ is relatively compact in $L^\infty(\bT^2\times\bS^1\times\bR_+\times[-1,1])$ weak-* and by 
the uniqueness of the limit point as $j\to+\infty$, we conclude that
$$
F_j/E\wto C=\tfrac1{2\pi}\iint_{\bT^2\times\bS^1}f^{in}(x,\om)dxd\om
$$
in $L^\infty(\bT^2\times\bS^1\times\bR_+\times[-1,1])$ weak-* as $j\to+\infty$. Thus, we have proved that
$$
K_jF_0\wto CE\hbox{ in }L^\infty(\bT^2\times\bS^1\times\bR_+\times[-1,1])\hbox{ weak-*}
$$
where
$$
\ba
C&=\tfrac1{2\pi}\iint_{\bT^2\times\bS^1}f^{in}(x,\om)dxd\om
\\
&=\tfrac1{2\pi}\iint_{\bT^2\times\bS^1}\int_0^\infty\int_{-1}^1F_0(x,\om,s,h)dhdsdxd\om\,,
\ea
$$
whenever $\grad^m_xf^{in}\in L^\infty(\bT^2\times\bS^1)$ for each $m\ge 0$, with $f^{in}\ge 0$ a.e. on $\bT^2\times\bS^1$. 

\smallskip
\noindent
\underline{Step2:}

The same holds true if $0\le f^{in}\in L^\infty(\bT^2\times\bS^1)$ without assuming that $\grad^m_xf^{in}\in L^\infty(\bT^2\times\bS^1)$ 
for each $m\ge 1$, by regularizing the initial data in the $x$-variable. 

Indeed, if $(\zeta_\eps)_{\eps>0}$ is a regularizing sequence in $\bT^2$ such that $\zeta_\eps(-z)=\zeta_\eps(z)$ for each $z\in\bT^2$, 
one has
$$
\ba
\iint_{\bT^2\times\bS^1}\int_0^\infty\int_{-1}^1(K_t(\zeta_\eps\star_xF_0)-K_tF_0)(x,\om,s,h)\phi(x,\om,s,h)dhdsdxd\om
\\
=
\iint_{\bT^2\times\bS^1}\int_0^\infty\int_{-1}^1K_t(F_0)(x,\om,s,h)(\zeta_\eps\star_x\phi-\phi)(x,\om,s,h)dhdsdxd\om
\ea
$$
because $K_t$ commutes with translations in the variable $x$. Since 
$$
\|K_tF_0/E\|_{L^\infty(\bT^2\times\bS^1\times\bR_+\times[-1,1])}=\|F_0/E\|_{L^\infty(\bT^2\times\bS^1\times\bR_+\times[-1,1])}
$$
for all $t\ge 0$, and 
$$
\zeta_\eps\star_x\phi-\phi\to 0\,\hbox{Êin }L^1(\bT^2\times\bS^1\times\bR_+\times[-1,1])\,,
$$
we conclude that
$$
K_t(\zeta_\eps\star_xF_0)-K_tF_0\wto 0\hbox{ in }L^\infty(\bT^2\times\bS^1\times\bR_+\times[-1,1])\hbox{ weak-*}
$$
as $\eps\to 0^+$ uniformly in $t\ge 0$. Since we have established in Step 1 that
$$
K_j(\zeta_\eps\star_xF_0)\wto\tfrac1{2\pi}\iint_{\bT^2\times\bS^1}f^{in}(x,\om)dxd\om
$$
in $L^\infty(\bT^2\times\bS^1\times\bR_+\times[-1,1])$ weak-* for each $\eps>0$, we conclude by a classical double limit argument that
$$
K_j(F_0)\wto CE \hbox{ with }C=\tfrac1{2\pi}\iint_{\bT^2\times\bS^1}f^{in}(x,\om)dxd\om
$$
in $L^\infty(\bT^2\times\bS^1\times\bR_+\times[-1,1])$ weak-* as $j\to+\infty$.

Summarizing, we have proved that
$$
K_jF_0\wto\tfrac1{2\pi}\left(\iint_{\bT^2\times\bS^1}\int_0^\infty\int_{-1}^1F_0(x,\om,s,h)dhdsdxd\om\right) E
$$
in $L^\infty(\bT^2\times\bS^1\times\bR_+\times[-1,1])$ weak-* as $j\to+\infty$. Replacing $F_0$ with $K_tF_0$ for each $t\in[0,1]$
and noticing that
$$
\ba
\iint_{\bT^2\times\bS^1}\int_0^\infty\int_{-1}^1&K_tF_0(x,\om,s,h)dhdsdxd\om
\\
&=
\iint_{\bT^2\times\bS^1}\int_0^\infty\int_{-1}^1F_0(x,\om,s,h)dhdsdxd\om
\ea
$$
we conclude that
$$
K_tF_0\wto\tfrac1{2\pi}\left(\iint_{\bT^2\times\bS^1}\int_0^\infty\int_{-1}^1F_0(x,\om,s,h)dhdsdxd\om\right) E
$$
in $L^\infty(\bT^2\times\bS^1\times\bR_+\times[-1,1])$ weak-* as $t\to+\infty$. 


\section{Proof of Theorem \ref{T-NoSpecGap}}


The argument used in the proof is reminiscent of the one used in \cite{Golse2003,Golse2008}.

Assume the existence of a profile $\Phi(t)$ such that the estimate (\ref{ExpEst}) holds. Therefore, for each initial data 
$f^{in}\in L^2(\bT^2\times\bS^1)$, the solution 
$$
F(t,\cdot,\cdot,\cdot,\cdot)=K_t(f^{in}(\cdot,\cdot)E(\cdot,\cdot))
$$ 
of the Cauchy problem satisfies
$$
\ba
\|F(t,\cdot,\cdot,\cdot,\cdot)\|_{L^2(\bT^2\times\bS^1\times\bR_+\times[-1,1])}
&\le
\|\la f^{in}\ra E\|_{L^2(\bT^2\times\bS^1\times\bR_+\times[-1,1])}
\\
&+\Phi(t)\|F(0,\cdot,\cdot,\cdot,\cdot)\|_{L^2(\bT^2\times\bS^1\times\bR_+\times[-1,1])}
\\
&=\la f^{in}\ra\sqrt{2\pi}\|E\|_{L^2(\bR_+\times[-1,1])}
\\
&+\Phi(t)\|f^{in}\|_{L^2(\bT^2\times\bS^1)}\|E\|_{L^2(\bR_+\times[-1,1])}\,.
\ea
$$
Assume $f^{in}\ge 0$ a.e. on $\bT^2\times\bS^1$; then $F\ge 0$ a.e. on $\bR_+\times\bT^2\times\bS^1\times\bR_+\times[-1,1]$, so 
that the right-hand side of the equation satisfied by $F$ is a.e. nonnegative on $\bR_+\times\bT^2\times\bS^1\times\bR_+\times[-1,1]$.
Thus $F\ge G$ a.e. on $\bR_+\times\bT^2\times\bS^1\times\bR_+\times[-1,1]$, where $G$ is the solution of the Cauchy problem
$$
\left\{
\ba
{}&(\d_t+\om\cdot\grad_x-\d_s)G(t,x,\omega,s,h)=0\,,
\\
&G(0,x,\om,s,h)=f^{in}(x,\om)E(s,h)\,.
\ea
\right.
$$
The Cauchy problem above can be solved by the method of characteristics, which leads to the explicit formula
$$
G(t,x,\om,s,h)=f^{in}(x-t\om,\om)E(s+t,h)\,,
$$
for a.e. $(t,x,\om,s,h)\in\bR_+\times\bT^2\times\bS^1\times\bR_+\times[-1,1]$. Thus
$$
\ba
\|F(t,\cdot,\cdot,\cdot,\cdot)\|&_{L^2(\bT^2\times\bS^1\times\bR_+\times[-1,1])}
\\
&\ge
\|G(t,\cdot,\cdot,\cdot,\cdot)\|_{L^2(\bT^2\times\bS^1\times\bR_+\times[-1,1])}
\\
&=
\|f^{in}\|_{L^2(\bT^2\times\bS^1)}\left(\int_t^\infty\int_{-1}^1E(s,h)^2dhds\right)^{1/2}
\\
&\ge
\|f^{in}\|_{L^2(\bT^2\times\bS^1)}\tfrac1{\sqrt{2}}\left(\int_t^\infty\left(\int_{-1}^1E(s,h)dh\right)^2ds\right)^{1/2}\,.
\ea
$$

Therefore, if there exist a profile $\Phi(t)$ satisfying the estimate in the statement of the theorem, one has
$$
\ba
\left(\int_t^\infty\left(\int_{-1}^1E(s,h)dh\right)^2ds\right)^{1/2}
\le\left(\frac{\sqrt{2\pi}\la f^{in}\ra}{\|f^{in}\|_{L^2(\bT^2\times\bS^1)}}+\Phi(t)\right)\sqrt{2}\|E\|_{L^2(\bR_+\times[-1,1])}
\ea
$$
for each $t>0$, and for each $f^{in}\in L^2(\bT^2\times\bS^1)$ s.t. $f^{in}\ge 0$ a.e. on $\bT^2\times\bS^1$.

Let $\rho\in C(\bR^2)$ such that 
$$
\rho\ge 0\,,\quad\Supp(\rho)\subset(-\tfrac14,\tfrac14)^2\,,
$$
and set
$$
\rho_\eps(x):=\rho\left(\frac{x}{\eps}\right)
$$
for each $\eps\in (0,1)$. Define $\tilde\rho_\eps$ as the periodicized bump function 
$$
\tilde\rho_\eps(x):=\sum_{k\in\bZ^2}\rho_\eps(x+k)\,.
$$
Clearly
$$
\ba
{}&\|\tilde\rho_\eps\|_{L^1(\bT^2)}=\|\rho_\eps\|_{L^1(\bR^2)}=\eps^2\|\rho\|_{L^1(\bR^2)}\,,
\\
&\|\tilde\rho_\eps\|_{L^2(\bT^2)}=\|\rho_\eps\|_{L^2(\bR^2)}=\eps\|\rho\|_{L^2(\bR^2)}\,.
\ea
$$
Choosing $f^{in}(x,\om):=\rho_\eps(x)$ in the inequality above leads to
$$
\ba
\left(\int_t^\infty\left(\int_{-1}^1E(s,h)dh\right)^2ds\right)^{1/2}
\le\left(\eps\frac{\|\rho\|_{L^1(\bR^2)}}{\|\rho\|_{L^2(\bR^2)}}+\Phi(t)\right)\sqrt{2}\|E\|_{L^2(\bR_+\times[-1,1])}
\ea
$$
and, letting $\eps\to 0^+$, we conclude that
$$
\left(\int_t^\infty\left(\int_{-1}^1E(s,h)dh\right)^2ds\right)^{1/2}\le\sqrt{2}\Phi(t)\|E\|_{L^2(\bR_+\times[-1,1])}\,.
$$
That $\Phi(t)=o(t^{-3/2})$ as $t\to+\infty$ is in contradiction with statement 3) in Theorem \ref{T-PropEquil}, which implies that 
$$
\left(\int_t^\infty\left(\int_{-1}^1E(s,h)dh\right)^2ds\right)^{1/2}\sim\frac1{\sqrt{3}\pi^2}\frac1{t^{3/2}}
$$
as $t\to+\infty$, while statement 1) in the same theorem implies that
$$
\|E\|_{L^2(\bR_+\times[-1,1])}^2\le \int_0^\infty\int_{-1}^1E(s,h)E(0,h)dhds=\int_0^\infty\int_{-1}^1E(s,h)dhds=1\,.
$$
 

\end{document}